\documentclass[11pt,a4paper,titlepage]{article}
\pdfoutput=1

\usepackage{graphicx}
\usepackage[pdftex]{hyperref}
\usepackage[english]{babel}
\usepackage{amsmath,amssymb}  
\usepackage{amsthm}          
\usepackage[usenames]{color}
\usepackage{tikz}
\usepackage[lofdepth,lotdepth]{subfig}
\usepackage{mathrsfs}
\usepackage{makeidx}
\usepackage{enumitem}
\usepackage{subfig}
\usepackage{subfloat}
\usepackage{nccmath}
\usepackage{color}

\usetikzlibrary{matrix,arrows}

\numberwithin{equation}{section}

\theoremstyle{plain}                       
\newtheorem{theorem}{Theorem}[section]     
\newtheorem{lemma}[theorem]{Lemma}          
\newtheorem{proposition}[theorem]{Proposition}
\newtheorem{corollary}[theorem]{Corollary}

\theoremstyle{definition}                    
\newtheorem{definition}[theorem]{Definition}  
\newtheorem{example}[theorem]{Example}

\theoremstyle{remark}                         
\newtheorem{remark}[theorem]{Remark}

\newcommand{\field}[1]{\mathbb{#1}} 
\newcommand{\R}{\field{R}} 
\newcommand{\N}{\field{N}} 
\newcommand{\C}{\field{C}} 
\newcommand{\Z}{\field{Z}}
\newcommand{\Q}{\field{Q}} 
\newcommand{\basisB}{\mathcal{B}}

\newcommand{\topX}{\mathscr{X}}
\newcommand{\topY}{\mathscr{Y}}
\newcommand{\topD}{\mathscr{D}}
\newcommand{\topDP}{\mathscr{D}^*}

\newcommand{\topC}{\mathscr{C}}
\newcommand{\neut}{\nu}

\newcommand{\acP}{\hat P}
\newcommand{\acE}{\hat E}

\newcommand{\Id}{\mathrm{Id}}

\newcommand{\li}{\mathrm{li}}

\newcommand{\LambdaE}[1]{\Lambda^{\hspace{-1 pt}#1}}
\newcommand{\LambdaU}[2]{\Lambda_{#2}^{\hspace{-1 pt}#1}}
\newcommand{\prodc}{\hat C}
\newcommand{\seqc}{\hat c}
\newcommand{\comments}[1]{}
\newcommand{\alt}[1]{#1'}
\newcommand{\shortmin}{\text{-}}
\newcommand{\mop}{\hspace{-1 pt}}
\newcommand{\inv}{\shortmin 1}
\newcommand{\minelt}[1]{\underline{#1}}
\newcommand{\maxelt}[1]{\overline{#1}}

\makeindex
\title{Expansion Systems}
\date{September 16, 2011}
\author{Victor Pessers}
\begin{document}

\selectlanguage{english}

\begin{titlepage}
\null\vfill
\begin{center}
\normalfont {\LARGE Expansion Systems\par}\bigskip {\Large Victor Pessers\par}
\bigskip {\Large 29 September 2011\par}
\end{center} 
\vfill \vfill
\begin{center}
{\Large Master Thesis}\par\bigskip {Supervisor:
Prof. Dr. A.J. Homburg}\par\medskip
\end{center}
\vfill
\begin{center}
\begin{equation*}
1+\displaystyle\int_0^{x}\cfrac{x_1\;\;dx_1}{
1-\displaystyle\int_0^{x_1}\cfrac{x_2\;\;dx_2}{
3+\displaystyle\int_0^{x_2}\cfrac{7\;\;x_3\;\;dx_3}{
5-\displaystyle\int_0^{x_3}\cfrac{\quad 221\;\;x_4\;\;dx_4\quad}{2205\;\cdot\;\text{etc.}}}}}
\end{equation*}
\end{center}
\vfill
\begin{center}
\leavevmode\normalfont KdV Institute for mathematics\par\smallskip University of
Amsterdam\par\smallskip\medskip
\includegraphics[width=0.075\hsize]{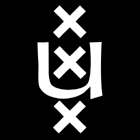}
\end{center}
\textcolor{white}{.}
\end{titlepage}

\begin{abstract}
In this thesis, we introduce {\em expansion systems} as a general framework to describe a large variety of approximation algorithms, such as Taylor approximation, decimal expansion and continued fraction. We consider some basic properties of expansion systems, and also study criteria for convergence. Further, we introduce the notion of isomorphisms between expansion systems. In the appendix, we discuss another class of expansion systems, which we call approximation systems. Many claims of convergence in this appendix, remain to be proven.
\end{abstract}

\tableofcontents

\newpage
\section*{Introduction}\label{SC introduction}
\addcontentsline{toc}{section}{Introduction}

The concept of approximation is an important concept throughout the realm of mathematics. Many different approximation algorithms have been developed, which are used in various contexts. But despite the variety of algorithms, there seem to be particular mechanisms which some of these algorithms have in common. For instance, many approximations are based on the existence of a fixed point for a certain contractive mapping, so that under the right conditions, iterated application of this mapping will approximate this fixed point. Much theory has been developed to capture this phenomenon, and its theorems are widely applicable (see for instance \cite{Ja}).

In this master thesis we will focus on a different approximation mechanism, which also seems to underlie many known approximation algorithms. It goes about as follows: we have an object $y$, on which we will apply a sequence of mappings. After each of these mappings, our $y$ has changed, and we record the essential information of this result before going to the next mapping (for example we repeatedly differentiate a function, and after each time we record its value in a point $x_0$). Then, after having recorded these values up to a certain level, we try to reverse process all these mappings, using just the data we have collected so far (for example we repeatedly integrate a function from the point $x_0$, using the recorded values in $x_0$ as the constants of integration). In many examples, this algorithm turns out to give good approximations for the initial object $y$. The more of the recorded coefficients we involve in our approximation, the more accurate it generally becomes. This observation should motivate further investigation to this general principle, and we will refer to these algorithmic structures as {\em expansion systems}.

Besides the example of Taylor series, which (as the reader may have figured out) just served to illustrate the underlying mechanism of expansion systems, there are many more approximation algorithms that can be viewed in this perspective. There is the decimal expansion, Fourier series, continued fraction, Engel expansion, Newton's forward difference formula, and so on. We will come across all of these examples as the theory of expansion systems is being unfolded.\\

In Section 1 of the thesis, a formal description of expansion systems is given, which will serve as the main framework for all later sections. 
In Section 2, we will discuss some basic features of expansion systems regarding the approximations it provides. As we will see, the objects that arise as approximations will actually have much more meaning in the context of expansion systems, than just being an approximation. In Section 3, we will develop the theory which is essential for investigating the convergence of expansion system. We will also prove some elementary convergence theorems, as well as a convergence result for the class of {\em monotonic expansion systems}. The subject of Section 4 is isomorphisms between different expansion systems. The concept of isomorphisms may help us to translate results from one expansion system to another, whenever such an isomorphism can be established.

In the final appendix, we will shed some light on a type of approximations, called {\em approximation systems}. In a somewhat different setting, these type of approximations have been the subject of the author's bachelor thesis {\em Approximation Systems} (in Dutch: Benaderingsstelsels)\cite{Pe}, and they are also the subject of the identically named  paper\cite{PeKo}, coauthored by Prof. Koornwinder. In this thesis, we will reformulate this type of approximations in the context of expansion systems. As of yet, there is still a discrepancy between the results regarding approximation systems that have been establish theoretically, and the results that have been experimentally obtained. For the treatment of the subject in the appendix, I have decided to especially share the more practical results, accepting the lack of foundation to prove them.\\

The subject of this thesis came into being during work on the article {\em Approximation Systems}. Prof. Koornwinder asked me to find out whether other types of approximation algorithms could somehow be fitted into the theory of approximation systems, and after much puzzling I decided that this was going to be difficult. However, there were nonetheless some basic principles which the approximation systems shared with other types of algorithms, and it became clear that there is in fact a general framework which all of these algorithms would fit in. When I decided to work out this idea as the subject of my master thesis, it was not quite clear to me yet, how much more I would be able to write down, besides the very definition and a couple of examples. But along the way, it appeared that there was actually quite a lot of theory to explore, and an initial habit of searching for content, soon enough changed into a habit of focusing on just a couple of basic features. I can therefore only apologize for all possible omissions, that probably would better have been included into this thesis. For example, like Prof. Homburg has remarked, it would make sense also to discuss the situation in which the expansions are linear maps. Although some of the examples in this thesis do belong to this category, such as the Taylor expansion and the Fourier series, a general theory for linear expansion systems is not given. And as Prof. Koornwinder suggested, also the theory of fast wavelet transforms could possibly be linked to expansion systems, but this hasn't made it into the final thesis either. Hopefully, future efforts will fill in gaps like these.

\newpage
\subsubsection*{Acknowledgments}

To begin with, I want to acknowledge Prof. Dijkstra and Prof. van Mill for their help with some topological issues. They have been of great support when I got stuck with a topological problem. Further, I want to acknowledge Prof. Koornwinder. After all, it have been the fruitful conversations with him, that have been the main inspiration for this thesis, and I thank him for the great care with which he has guided me in the last couple of years. Finally, I want to thank my supervisor Prof. Homburg. Besides all his other help, I am very grateful for his willingness and the freedom he has given me, to let me work on something what was actually just a spontaneous idea.

\newpage
\section{Description of Expansion Systems}

We consider the following situation: we have a sequence of so called {\em element spaces}\index{Element!space} $S = \{S_i\}_{i\in \N}$, where $\N={0,1,\ldots}$, and in each set $S_i$ there is a {\em neutral element}\index{Element!neutral} $\neut_i\in S_i$\index{zsymbols@\textbf{Symbols}!nu@$\neut_i$}. For instance, if our $S_i$ are additive spaces, this will often be the zero element.
Further we have another sequence of so called {\em coefficient spaces}\index{Coefficient!space} $C = \{C_i\}_{i\in \N}$\index{zsymbols@\textbf{Symbols}!S@$S,S_i,C,C_i$}. Between these sets, we have the following collection of injections: $F_i = (P_i,E_i)$\index{zsymbols@\textbf{Symbols}!E@$E,E_i,P,P_i,F,F_i$}
, where $P_i\colon S_i \rightarrow C_{i}$ and $E_i\colon S_i \rightarrow S_{i+1}$, mapping the neutral element of $S_i$ to the neutral element of $S_{i+1}$, i.e.
\begin{equation}\label{EQ neut elt}
E_i(\neut_i) = \neut_{i+1}.
\end{equation}
Note that $P_i$ and $E_i$ themselves don't need to be injective. In diagram this looks like shown in Figure \ref{FI structure es}.

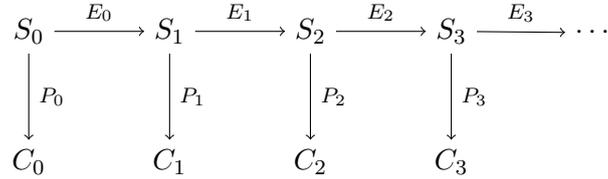
\begin{figure}[h]
\centering
\begin{tikzpicture}[descr/.style={fill=white,inner sep=2.5pt}]
 \matrix (m) [matrix of math nodes, row sep=3em,
 column sep=3em]
 { S_0 & S_1 & S_2 & S_3 & \cdots \\
   C_0 & C_1 & C_2 & C_3 &  \\ };
 \path[->,font=\scriptsize]
 (m-1-1) edge node[auto] {$ E_0 $} (m-1-2)
         edge node[auto] {$ P_0 $} (m-2-1)
 (m-1-2) edge node[auto] {$ E_1 $} (m-1-3)
         edge node[auto] {$ P_1 $} (m-2-2)
 (m-1-3) edge node[auto] {$ E_2 $} (m-1-4)
         edge node[auto] {$ P_2 $} (m-2-3)
 (m-1-4) edge node[auto] {$ E_3 $} (m-1-5)
         edge node[auto] {$ P_3 $} (m-2-4);
 \end{tikzpicture}
\caption{Structure of expansion system}
\label{FI structure es}
\end{figure}
Such a system~$$\Psi=(S,C,F):=(\{S_i\}_{i\in \N},\{C_i\}_{i\in \N},
\{F_i\}_{i\in \N})$$\index{zsymbols@\textbf{Symbols}!Psi@$\Psi$}
we call an {\em expansion system}\index{Expansion system}, or ES for short.\footnote{Whenever we refer to a sequence of element spaces $S$, we implicitly assume the existence of a sequence $\neut_i$ of neutral elements, even when not mentioned explicitly what these neutral elements are.} Each $P_i$ in our ES will typically function as a projection on some simpler space, such as $\Z$ or $\C$. The mapping $E_i$ typically maps $S_i$ onto a similar space $S_{i+1}$, but thereby losing some information which is contained in the value of $P_i$.
\begin{example}\label{EX dec exp 1}
For $i\in \N$ let $(S_i,\neut_i)=([0,1),0)$, the set of real numbers greater or equal to $0$ and smaller than $1$, with $0$ being the neutral element. Let $C_i=\{0,1,\ldots,9\}$. Furthermore, let $P_i\colon S_i\rightarrow C_i$ be the projection onto the first digit after the dot (i.e. $P_i\,y:=\lfloor 10 y\rfloor$, where $\lfloor\cdot\rfloor$ denotes the floor function), and let $E_i\colon S_i\rightarrow S_{i+1}$ be defined as $E_i\,y := 10y-P_i\,y$. This expansion satisfies \eqref{EQ neut elt}. Now $F_i:=(P_i,E_i)$ is injective, for the equations
$$
\lfloor 10 y\rfloor=\lfloor 10 y'\rfloor,\quad\text{and}\quad
  10y-\lfloor 10 y\rfloor=10y'-\lfloor 10 y'\rfloor,
$$
imply of course that $y=y'$. Hence $\Psi = (S,C,F)$ defines an expansion system.
\end{example}
Note that in the example above, when we apply $E_i$ to some element $y\in S_i$, we have lost some information about $y_i$ in the value $E_i\,y$, namely the value of its first digit after the dot. This information is then contained in the value $P_i\,y$. But despite our loss of information, $E_i$ really is a surjective mapping from $S_i$ to $S_{i+1}$, which spaces are both equal to $[0,1)$. This is possible because of the expanding nature of $E_i$. As this is the typical pattern in most ES's, we will refer to the mappings $P_i$ as our {\em projections}\index{Projection}, and to the mappings $E_i$ as our {\em expansions}\index{Expansion}. However, the mappings $P_i$ and $E_i$ do not necessarily have to obey this particular pattern, and the words projections and expansions are therefore merely terminology.\\

Because all $F_i$ are injective, we also have a sequence of inverses {${F_i}^{\inv}\colon F_i(C_i\times S_i)\rightarrow S_i$}. Now given an $y\in S_0$, we define $y_0:=y$ and $y_{i+1}:=E_i\,y_i$\index{zsymbols@\textbf{Symbols}!y@$y,y_i,c,c_i$}, and furthermore $c_i:=P_i\,y_i$, so that we get $y_i = {F_i}^{\inv}(y_{i+1},c_i)$. These $c_i$ are called {\em coefficients}\index{Coefficient}, and the full sequence $\seqc:=(c_0,c_1,\ldots)$ is called the {\em coefficient code}\index{Coefficient!code}.\index{zsymbols@\textbf{Symbols}!c_seq@$\seqc$}

\begin{remark}
For convenience, we will adopt the special notation $\acE_i$\index{zsymbols@\textbf{Symbols}!E1@$\hat E_i,\hat P_i$} for the accumulative composition of the first $i$ expansions, so that we have $\acE_i\,y = E_i \cdots E_0\,y = y_i$. Likewise, we write $\acP_i$ for the composition of $P_i$ with $\acE_i$, so that we have $\acP_i\,y = P_i\,E_i \cdots E_0\,y = c_i$. Also to the $\acP_i\colon S_0 \rightarrow C_i$, we will refer to as projections, and the Cartesian product of all these projections gives rise to the so called {\em coefficient map}\index{Coefficient!map}
$$\acP\colon S_0 \rightarrow \prodc,\ \ \ y \mapsto (c_0,c_1,c_2,\ldots),$$
where $\prodc:=\prod_{i\in\N} C_i$\index{zsymbols@\textbf{Symbols}!C_prod@$\prodc$} stands for the Cartesian product of all the coefficient spaces. In a certain sense, this map codifies elements of $S_0$ as a list of coefficients from the spaces $C_i$, but it does so not necessarily injective.
\end{remark}

Having settled down the prerequisites of an ES $\Psi$, it turns out that under the right conditions 
our $\Psi$ gives rise to an approximation $y^{[n]}$ for $y\in S$, based on its coefficients $\seqc=(c_0,c_1,\ldots)=\acP\,y$. This approximation is constructed as follows:
\begin{align}
y_n^{[n]}&:=\neut_n\\
y_i^{[n]}&:={F_i}^{\inv}(c_i,y_{i+1}^{[n]})\qquad (i = n-1,n-2,\ldots,0)
\end{align}\index{zsymbols@\textbf{Symbols}!y1@$y^{[n]},y_i^{[n]}$}
provided that
\begin{equation}\label{EQ F proper}
(c_{i},y_{i+1}^{[n]})\in F_{i}(S_{i}),\qquad\text{for\ }i=n-1,n-2,\ldots,0.
\end{equation}
Our ultimate $n$-th order approximation will then be:
\begin{equation}
y^{[n]} := y_0^{[n]}
\end{equation}

The above construction of $y^{[n]}$, leads us to the following definitions:

\begin{definition}
An expansion system $\Psi$ is called {\em proper}\index{Expansion system!proper} at order $n$ w.r.t. $y\in S_0$ (or equivalently, to the coefficients $c_0,c_1,\ldots,c_{n-1}$), if it satisfies \eqref{EQ F proper} throughout the process of constructing $y^{[n]}$. Only when dealing with an ES which is proper at order $n$, this construction will succeed, and are we allowed to talk about the $n$-th order approximation of $y$, also called the $n$-th {\em convergent}\index{Element!convergent (object)} of $y$. When $\Psi$ is proper at any order $n\comments{\in \N}$ w.r.t. $y$, we say $\Psi$ is (completely) proper w.r.t. $y$. Also, when we say that $\Psi$ is proper without mentioning w.r.t. which element, it is assumed that it is proper w.r.t. any element in $S_0$. As this concept will return quite frequently, we will make use of special notation for the set of elements w.r.t. which $\Psi$ is proper:
\begin{equation}
S_0^{(n)}:=\{y\in S_0\,:\,\text{$\Psi$ is proper at order $n$ w.r.t. $y$}\}.
\end{equation}\index{zsymbols@\textbf{Symbols}!S1@$S_0^{(n)},S_0^{(\leq n)},S_0^{(< \infty)}$}
Likewise, we let $S_0^{(\leq n)}$ and $S_0^{(< \infty)}$ denote the elements in $S_0$, w.r.t. which $\Psi$ is proper up to and including order $n$, respectively completely proper.
\end{definition}

In situations where we are dealing with multiple ES's at the same time, confusion might arise about relative to which ES we are determining the approximating $y^{[n]}$. In that case we will write $y^{[\Psi,n]}$\index{zsymbols@\textbf{Symbols}!y1psi@$y^{[\Psi,n]},y_i^{[\Psi,n]}$} instead, to stress the ES we are working with.

We will end this section with a simple proposition about properness, and give some applications to illustrate its use.

\begin{proposition}\label{PR bijective F}
Suppose $\Psi$ is such that all $F_i\colon S_i\rightarrow S_{i+1}\times C_i$ are bijective. Then $\Psi$ is proper.
\end{proposition}
\begin{proof}
We have for $i=n,n-1,\ldots,1$ that, given $y_i^{[n]}\in S_i$ and $c_{i-1}\in C_{i-1}$, the surjectivity of $F_{i-1}$ implies $(c_{i-1},y_i^{[n]})\in F_{i-1}(S_{i-1})$. Hence \eqref{EQ F proper} is always satisfied.
\end{proof}

Note that we could just as well have required $F_i$ to be surjective, as we already have that it is injective by definition of an ES. This proposition implies for instance properness of the ES $\Psi$ in Example \ref{EX dec exp 1}. To see this, let's consider this ES once more here.

\begin{example}[\bf Decimal Expansion]\label{EX dec exp 2}
Let $\Psi$ be as in Example \ref{EX dec exp 1}. It is proper by Proposition \ref{PR bijective F}, as we have that each $F_i$ is bijective, and its inverse is given by:
$$
F_i^{\inv}(c,y)=\frac{c+y}{10}.
$$
So if the coefficients of some $y\in S_0$ are $c_0,c_1,\ldots$, then we get as $n$-th order approximation:
\begin{align*}
y_n^{[n]} &= 0\\
y_{n-1}^{[n]} &= \frac{c_{n-1}}{10}\\
y_{n-2}^{[n]} &= \frac{c_{n-2}}{10}+\frac{c_{n-1}}{100}\\
\ldots\\
y_0^{[n]} &= \frac{c_{0}}{10} + \frac{c_1}{100} + \ldots + \frac{c_{n-1}}{10^{n-1}}.
\end{align*}
In other words, the $n$-th order approximation coincides with the decimal number:
$$
y^{[n]}=0.c_0c_1c_2\ldots c_{n-1}
$$
It is good to realize that it really requires a proof that this particular sequence of decimal approximations is not only convergent, but that the number it converges to is also the number we used as input of the algorithm, so that indeed we may suggestively write:
$$
y = 0.c_0c_1c_2c_3\ldots
$$
Although an ad hoc proof of this wouldn't be that difficult to provide, we will postpone this to Theorem \ref{TH monotonic es}, where a more general result will be proven. Further it should be remarked that, although the space $S_0$ is restricted to the set of non-negative numbers smaller than $1$, we can quite easily adept the ES so that it can take bigger numbers as input. If we want $S_0$ to be $\R_{\geq 0}$, we just put $P_0\,y = \lceil\log_{10}\,y\rceil$ for non-zero $y$, and $E_0\,y = y/10^{P_0(y)}$ for non-zero $y$, and $F_0(0)=(0,0)$ for the case $y=0$. This extends the ES to non-negative numbers of arbitrary size.
\end{example}

\begin{example}[\bf Continued Fraction]\label{EX cont frac}
Just as in Example \ref{EX dec exp 1}, we let $(S_i,\neut_i)=([0,1),0)$, but now we let $C_i=\Z_{\geq 1}\cup\{\infty\}$. Furthermore, let $P_i$ be given by:
$$
P_i\,y = \begin{cases} \infty &\text{if}\quad y = 0\\
\lfloor 1/y\rfloor, &\text{otherwise}, \end{cases}
$$
and $E_i$ be given by
$$
E_i\,y = \begin{cases} 0 &\text{if}\quad y = 0\\
1/y-P_i\,y &\text{otherwise}. \end{cases}
$$
This satisfies Equation \ref{EQ neut elt}, and $F_i = (P_i,E_i)$ is a bijection, for which we have as inverse:
$$
F_i^{\inv}(c,y)= \begin{cases} 0 &\text{if}\quad c = \infty\\
\frac{1}{c+y} &\text{otherwise}. \end{cases}
$$
By Proposition \ref{PR bijective F}, $\Psi$ is proper and for its $n$-th approximation we find:
\begin{align*}
y_n^{[n]} &= 0\\
y_{n-1}^{[n]} &= \frac{1}{c_{n-1}}\\
y_{n-2}^{[n]} &= \cfrac{1}{c_{n-2}+\cfrac{1}{c_{n-1}}}\\
\ldots\\
y_0^{[n]} &= \cfrac{1}{c_0 + \cfrac{1}{c_1+\cfrac{1}{\ldots+\cfrac{1}{c_{n-1}}}}}
\end{align*}
It turns out that these rational approximations converge to the initial element $y$ as $n$ goes to infinity (cf. \cite[\S1.6]{BrSt}), which justifies the concise notation:
\begin{equation}\label{EQ cont frac}
y = c_0 + \cfrac{1}{c_1+\cfrac{1}{c_{2}+\cfrac{1}{c_3+\ldots}}}
\end{equation}
Just as for the decimal expansion, we will prove later on in this thesis that this is indeed valid for all $y\in S_0$.
\end{example}

\newpage
\section{Orders and Convergents}\label{SC orders}

In the previous section, we have seen how an expansion system codifies its elements in a sequence of coefficients. We have also discussed how these coefficients on its turn can give rise to an approximation.
In this section, we will see that the approximations are characterized by the property that they are eventually mapped by the expansion sequence $E$ to a neutral element $\neut_n\in S_n$. The values of its first $n$ coefficients will coincide with the element it approximates. We will also see that the existence of elements with these properties, is at the same time a guarantee of properness. Throughout all the propositions in this section, we will assume we are given some expansion system $\Psi$, and we again adopt the notation of the introduction for all of its components.

\begin{definition}\label{DF order}
We call $y\in S_0$ an element of {\em finite order}\index{Element!order}, if for certain $n\in\N$ we have $y_n = \neut_n$. The smallest such $n$, we call the order of $y$. If $y$ is not of finite order, we say it is of {\em infinite order}. As with the property of properness, we will use special notations $S_0^{[n]}$\index{zsymbols@\textbf{Symbols}!S2@$S_0^{[n]},S_0^{[\leq n]},S_0^{[< \infty]},S_0^{[\infty]}$}, $S_0^{[\leq n]}$, $S_0^{[< \infty]}$, and $S_0^{[\infty]}$ to denote the elements in $S_0$ respectively of order $n$, order less than or equal $n$, finite order, and infinite order.
\end{definition}

\begin{remark}\label{RM alt def order}
By \ref{EQ neut elt}, the order of an element $y$ could also be characterized by saying: $y$ has order $n$, iff $y_m = \neut_m$ for all $m\geq n$. This is equivalent to the definition above.
\end{remark}

\begin{theorem}\label{TH char prop}
We have $y\in S_0^{(n)}$, iff there is an element $y'\in S_0^{[n]}$, having the same coefficients up to and including order $n-1$, i.e. $\acP_i\,y' = \acP_i\,y = c_i$ for $i\leq n-1$. If this is the case, then $y'$ is unique and equals $y^{[n]}$.\footnote{Throughout this thesis, primes are only used to distinguish symbols (e.g. $y$ and $y'$), and do never have the meaning of differentiation.}
\end{theorem}
\begin{proof}
Proof of the forward implication: one verifies that it follows directly from the construction of $y^{[n]}$ that $E_i\,y_i^{[n]}=y_{i+1}^{[n]}$, and so we have for the composition $\acE_i$ of expansions:
\begin{equation}\label{EQ convergent}
(y^{[n]})_i = \acE_i\,y^{[n]} = E_i\cdots E_0\,y_0^{[n]} = y_i^{[n]}
\end{equation}
and thus in particular $(y^{[n]})_n = y_n^{[n]} = \neut_n$, which means $y^{[n]}\in S_0^{[n]}$.  Moreover, \eqref{EQ convergent} gives us that:
$$\acP_i\,y^{[n]} = P_i\,(y^{[n]})_i = P_i\,y_i^{[n]} = P_i\,{F_i}^{\inv}(y_{i+1}^{[n]},c_i) = c_i = \acP_i\,y.$$
Hence $y^{[n]}$ satisfies the requirements for our $y'$.\\
Now we prove the backward implication. We just carry out the construction of $y^{[n]}$, and let $y_n^{[n]}=\neut_n$. Because we know that $y'_{n-1}$ is an element in $S_{n-1}$ for which $F_{n-1} \,y'_{n-1} = (\neut_n,c_{n-1})$, it follows that $y_{n-1}^{[n]}={F_i}^{\inv}(y_{n}^{[n]},c_{n-1})$ is well defined, and equals $y'_{n-1}$. The same observation now applies to $F_{n-1}\,y'_{n-2} = (y'_{n-1},c_{n-2})$, so that $y_{n-2}^{[n]} = y'_{n-2}$. Continuing this process, we get at the last step that $y_{0}^{[n]} = y^{[n]}$ is well defined and equals $y'_0 = y'$. Because of the arbitrarity of $y'$, and the fact that $y_{0}^{[n]}$ is uniquely determined by its construction, the additional remark follows that this $y'$ is unique.
\end{proof}

\begin{corollary}\label{CR conv wrt fin order}
Given $y\in S_0^{[n]}$, we have that $y\in S_0^{(\geq n)}$, and $y^{[m]}=y$ for all $m\geq n$.
\end{corollary}
\begin{proof}
This follows from the observation that by Remark \ref{RM alt def order} $y_m = \neut_m$ for all $m\geq n$. So for each order $m\geq n$, $y$ itself serves as the looked after element $z$, and so we have $y^{[m]}=y$.
\end{proof}

\begin{remark}\label{RM impl theor}
Theorem \ref{TH char prop} not only provides another formulation for properness, it also characterizes the $n$-th convergent $y^{[n]}$ of an element $y$ as the unique $n$-th order element $z$ in $S_0$, having the same coefficients as $y$ up to and including order $n-1$. The corollary on its turn, implies that
$$S_0^{[n]} = \{y^{[n]}\,:\,y\in S_0^{(n)}\}.$$
In particular we have that $S_0^{[n]}\subset S_0^{(\geq n)}$.
\end{remark}

We can apply Theorem~\ref{TH char prop} to the following example of an expansion system:

\begin{example}{\rm\bf Taylor Series}\label{EX taylor}\\
Let each $S_i$ be the space of holomorphic functions on some open neighborhood $U\subset \C$ of a point $x_0\in \C$, and let $C_i = \C$ and $\neut_i$ the zero function. Furthermore, given a function $y(x)\in S_i$, let $P_i\,y = y(x_0)$ and
$$E_i y(x) = \frac{y(x)-P_i\,y}{x-x_0},$$
the value in $x_0$ being defined by its limit. One verifies that for given $y(x) = c_0 + c_1 (x-x_0) + \ldots$, we have $\acP_i\,y = c_i$. We have that $F_i=(P_i,E_i)$ satisfies \ref{EQ neut elt}, and
furthermore $\Psi=(S,C,F)$ is proper, for given an order $n$, we have that $z = c_0 + c_{1} (x-x_0) + \ldots + c_{n-1} (x-x_0)^{n-1}$ satisfies Theorem~\ref{TH char prop}. (We could also have invoked Proposition \ref{PR bijective F}, noting that every holomorphic function has an integral with prescribed value in the point $x_0$.) In this ES, the elements of order $n$ are just the polynomials of order $n-1$ (the zero polynomial in case $n=0$). The elements of infinite order are exactly all non-polynomial holomorphic functions. The $n$-th order approximation coincides with the $n-1$-th Taylor approximation at the point $x_0$.
\end{example}

It is important to realize that if an ES $\Psi$ is proper w.r.t. $y$ at order $n$, this doesn't necessarily mean that $\Psi$ is also proper w.r.t. $y$ at orders below $n$, nor that it is proper at orders above $n$. We will use a variation on Example \ref{EX taylor} as a counterexample.

\begin{example}\label{EX taylor count ex}
Again we consider the expansion system $\Psi$ as given in Example \ref{EX taylor}, but this time we let $x_0 = 0$ and $U$ be the closed interval $[0,1]$. Further, we let $\neut_i$,$C_i$,$P_i$ and $E_i$ just be as in Example \ref{EX taylor}, but we restrict each $S_i$ only to those elements $y_i$, such that $\|y_j\|\leq 1$ for all $j\geq i$, i.e. $y_i$ remains in absolute value $\leq 1$ after composition with arbitrarily many expansion $E_j$, $j=i,i+1,\ldots$. In particular, the elements in $S_i$ itself are of absolute value $\leq 1$. Note that by this condition, we still have $E_i(S_i)\subset S_{i+1}$. Now consider the polynomial $y=\frac{1}{2}+x-x^2+x^3-x^4$. We can calculate that:
\begin{align*}
\|y_0\| &= \|\frac{1}{2}+x-x^2+x^3-x^4\|\approx 0.60583<1\\
\|y_1\| &= \|1-x+x^2-x^3\|=1\\
\|y_2\| &= \|-1+x-x^2\|=1\\
\|y_3\| &= \|1-x\|=1\\
\|y_4\| &= \|-1\|=1\\
\|y_5\| &= \|y_6\| = \ldots = \|0\| = 0.
\end{align*}
All these $y_i$ have absolute value $\leq 1$ on the interval $[0,1]$, and so $y_i\in S_i$. It turns out that $y\in S_0^{(3)}$, as we have for $z = \frac{1}{2} + x - x^2$:
\begin{align*}
\|z_0\| &= \|\frac{1}{2}+x-x^2\|= \frac{3}{4}<1\\
\|z_1\| &= \|1-x\|=1\\
\|z_2\| &= \|-1\|=1\\
\|z_3\| &= \|z_4\| = \ldots = \|0\| = 0.
\end{align*}
So $z$ is an element in $S_0^{[3]}$, and it satisfies Theorem \ref{TH char prop}, by which we may conclude properness at order $3$ w.r.t. $y$. But for order $n=2$, our only candidate is $z = \frac{1}{2}+x$, but $\|\frac{1}{2}+x\|=1\frac1 2>1$, so there can be no such $z$ in $S_0$. Likewise, for order $n=4$, we have $\frac{1}{2}+x-x^2+x^3$ as only candidate, but $\|\frac{1}{2}+x-x^2+x^3\|=1\frac1 2>1$, so this $z$ does not lie in $S_0$ either. Hence $\Psi$ is proper w.r.t. $y$ at order $3$, but not at order $2$ and $4$.
\end{example}

\begin{proposition}\label{PR eq coef}
Suppose $y\in S_0^{(n)}$, and $y'\in S_0$ has the same first $n$ coefficients as $y$, i.e. $\acP_i\,y =\acP_i\,y'$ for $0\leq i\leq n-1$. Then also $y'\in S_0^{(n)}$, and we have $y^{[n]}={y'}^{[n]}$.
\end{proposition}
\begin{proof}
By Theorem \ref{TH char prop} we conclude that there is an element $y^{[n]}\in S_0^{[n]}$ with the same first $n$ coefficients as $y$. So by assumption, $y^{[n]}$ also has the same first $n$ coefficients as $y'$. Therefore, again by Theorem \ref{TH char prop}, we have that $y'\in S_0^{(n)}$, and $y^{[n]}={y'}^{[n]}$.
\end{proof}

\begin{corollary}
Suppose for certain $n$ and $m\leq n$ we have $y\in S_0^{(n)}\cap S_0^{(m)}$, i.e. $\Psi$ is proper w.r.t. $y$ both at order $n$ and $m$. Then also $y^{[n]}\in S_0^{(n)}\cap S_0^{(m)}$, and we have $(y^{[n]})^{[m]} = y^{[m]}$, i.e. the $m$-th order convergent of $y^{[n]}$ equals the $m$-th order convergent of $y$ itself. In particular $(y^{[n]})^{[n]} = y^{[n]}$.
\end{corollary}
\begin{proof}
By Theorem \ref{TH char prop}, $y$ and $y^{[n]}$ have equal coefficients up to and including order $n-1$. Then certainly $y$ and $y^{[n]}$ have equal coefficients up to and including order $m-1$. Hence Proposition \ref{PR eq coef} applies, which implies that $y^{[n]}\in S_0^{(m)}$, and $(y^{[n]})^{[m]} = y^{[m]}$.
\end{proof}

\newpage
\section{Convergence of Expansion Systems}\label{SC convergence}

When we are given an ES $\Psi$ with a topological space $S_0$, it would be nice if $y^{[n]}$ becomes a better approximation of $y$ as $n$ increases, so that we can define $y^{[\infty]}:=\displaystyle\lim_{n\rightarrow \infty} y^{[n]} = y$. However, this may not always be the case, and the following definition provides names for the situations that may arise:

\begin{definition}
Suppose $\Psi$ is proper with respect to $y$ in the topological space $S_0$. We then call $\Psi$ {\em convergent}\index{Expansion system!convergent (property)} with respect to $y$, if $y^{[n]}$ converges to some element in $S_0$, for which we will use the notation $y^{[\infty]}$\index{zsymbols@\textbf{Symbols}!yinf@$y^{[\infty]}$}.
We say $\Psi$ is {\em accurate}\index{Expansion system!accurate} w.r.t. $y$, if it is not only convergent with some limit $y^{[\infty]}$, but moreover $y^{[\infty]}$ equals $y$. We say $\Psi$ is {\em divergent}\index{Expansion system!divergent} with respect to $y$, if the sequence of convergents $y^{[n]}$ does not converge at all. Furthermore, we simply say $\Psi$ is convergent or accurate, if it has this respective property with respect to all of its elements.
\end{definition}

\begin{remark}\label{RM conv wrt fin order}
Note that by Corollary \ref{CR conv wrt fin order}, an ES is always accurate w.r.t. its elements of finite order, even if $S_0$ is given the discrete topology.
\end{remark}

In this section we will investigate some of the conditions under which $\Psi$ is (accurately) convergent. As convergence depends on the topology which the element spaces are endowed with, it is important to pay attention to the topological aspects of ES's. As it will turn out, in many cases the approximation will converge in one topology, but not in the other. We will therefore first develop a framework which will help us in investigating these phenomena.

\subsection{Topologies on Expansion Systems}\label{SS top on es}

\begin{definition}
Let $f:X\rightarrow Y$ be a map between $X$ and $Y$, where $Y$ is given the topology $\topY$. Then we define the {\em generated topology}\index{Topology!generated} $f^{\inv}(\topY)$, to be the topology on $X$ generated by the map $f$, meaning it is generated by the sets $f^{\,\inv}(U)$, where $U$ is open in $\topY$ (cf. \cite[1.4.8]{Enk}).

More specific, let $\topC$\index{zsymbols@\textbf{Symbols}!Ctop@$\topC$} be a topology on the space $\prodc$ of coefficient sequences. We then define $\acP^{\,\inv}(\topC)$\index{zsymbols@\textbf{Symbols}!PinvC@$\acP^{\,\inv}(\topC)$} to be the topology on $S_0$, generated by the sets $\acP^{\,\inv}(U)$ with $U$ open in $\topC$, and call this the {\em generated topology}. Usually each space $C_i$ already has its own topology $\topC_i$, and unless otherwise specified, we will assume that $\topC$ is the product topology with respect to these topologies. In that case, we call the generated topology $\acP^{\,\inv}(\topC)$ the {\em weak topology}\index{Topology!weak} generated by the collection of projections $\acP_i$ to the spaces $(C_i,\topC_i)$ (cf. \cite[3.16]{ArPo}). In the context of expansion systems, the weak topology on $S_0$ which is generated by the discrete topologies on $C_i$ turns out to be an important topology. We will call this topology on $S_0$ the {\em weakly discrete topology}\index{Topology!weakly discrete}, and denote it by $\acP^{\,\inv}(\topDP)$\index{zsymbols@\textbf{Symbols}!Dtop@$\topD,\topDP$}, where $\topDP$ on its turn denotes the discrete product topology on $\prodc$. The discrete topology on $\prodc$ will be written as just $\topD$.
\end{definition}

\begin{remark}\label{RM topology}
The generated topology $\acP^{\,\inv}(\topC)$ can also be defined as the weakest topology such that $\acP\colon S_0\rightarrow \prodc$ is continuous. In particular, the weak topology is the weakest topology such that all $\acP_i\colon S_0\rightarrow C_i$ are continuous (this follows from the fact that $\acP$ is continuous iff each $P_i$ is continuous, see \cite[19.6]{Mk}). This topology is therefore generated by sets of the form $P_i^{\inv}(U_i)$, where $i\in \N$, and $U_i$ is open in $C_i$. This implies that a sequence $y^{[n]}$ in $S_0$ converges to $y$ in the weak topology, iff for each $i\in \N$: $\acP_i\,y^{[n]}$ converges to $\acP_i\,y$ in the topology $\topC_i$ of $C_i$. Equivalently, $y^{[n]}$ converges to $y$ in the weak topology, iff $\acP\,y^{[n]}$ converges to $\acP\,y$ in the product topology $\topC$. The weakly discrete topology $\acP^{\,\inv}(\topDP)$ is likewise generated by the sets $\acP_i^{\,\inv}(c_i)$, where $i\in \N$ and $c_i$ runs through all the points in $C_i$. In particular, $y^{[n]}$ converges to $y$ in $\acP^{\,\inv}(\topDP)$, iff for each $i\in \N$: $\acP_i\,y^{[n]}$ equals $\acP_i\,y$ for $n$ large enough. The weakly discrete topology is the strongest one among all possible weak topologies generated by the $\acP_i$, so given an arbitrary weak topology, the weakly discrete topology is necessarily stronger or equal to it. Although it is useful to use different names for the generated topology and the weak topology, we should keep in mind that in a strict sense there is no essential distinction between these two types of topology, because a weak topology can be viewed as the topology generated by the product map to the product topology, and a generated topology can be viewed as a weak topology for just one map.
\end{remark}

\begin{lemma}\label{LM basis top}
Consider an ES $\Psi$ where $\basisB_i\ (i\in \N)$ are bases for the topologies $\topC_i$ of $C_i$. Then the product topology $\topC$ on $\prodc$ has a basis of the form:
\begin{equation}\label{EQ basis top C}
\pi_n^{\inv}(B_0,\ldots,B_{n-1})=B_0\times\ldots\times B_{n-1}\times C_n\times C_{n+1}\times\ldots,
\index{zsymbols@\textbf{Symbols}!piinvB@$\pi_n^{\inv}(B_0,\ldots,B_{n-1})$}
\end{equation}
where $B_i\in \basisB_i$ and $\pi_n\colon \prodc\rightarrow \prod_{i< n}C_i$ is the projection on the first $n$ coordinates. The weak topology $\acP^{\inv}(\topC)$ on $S_0$ has as a basis:
\begin{equation}\label{EQ basis top S}
\acP^{\inv}(B_0,\ldots,B_{n-1}):=\bigcap_{i=0}^{n-1} \acP_i^{\inv}(B_i),\index{zsymbols@\textbf{Symbols}!PinvB@$\acP^{\,\inv}(B_0,\ldots,B_{n-1})$}
\end{equation}
where $B_i\in \basisB_i$. In particular, we have that $\topDP$ has a basis of the form:
\begin{equation}\label{EQ basis top DP}
\pi_n^{\inv}(c_0,\ldots,c_{n-1}) = \{c_0\}\times\ldots\times \{c_{n-1}\}\times C_n\times C_{n+1}\times\ldots,
\end{equation}
and $\acP^{\,\inv}(\topDP)$ has a basis:
\begin{equation}\label{EQ basis top WD}
\acP^{\inv}(c_0,\ldots,c_{n-1})=\bigcap_{i=0}^{n-1} \acP_i^{\inv}(c_i),
\end{equation}
where each $c_i$ is a point in $C_i$.
\end{lemma}
\begin{proof}
We recall that the product topology $\topC$ is generated by all sets of the form $\prod_{i\in\N} V_i$, where for only finitely many $i\in\N$ we have $C_i\neq V_i\in \basisB_i$ (see for instance\cite[Thm. 19.2]{Mk}\cite[2.3.1]{Enk}).
Given such a set $\prod_{i\in\N} V_i$, we may assume that $V_i = C_i$ for $i\geq n$. Now let $\seqc = (c_0,c_1,\ldots)\in \prod_{i\in\N} V_i$. For each $i<n$ we pick a basis element $B_i$ of $\basisB_i$, such that $c_i\in B_i\subset V_i$ (when $V_i=C_i$, any $B_i$ containing $c_i$ will suffice, and when $V_i\neq C_i$, we can just set $B_i=V_i$). Now $\pi_n^{\inv}(B_0,\ldots,B_{n-1})$ is an open set containing $\seqc$, and which lies within the set $\prod_{i\in\N} V_i$. By the arbitrarity of $\prod_{i\in\N} V_i$ and $\seqc$, this proves that \eqref{EQ basis top C} defines indeed a basis for $\topC$. Equation \eqref{EQ basis top S} follows from the observation that if $\basisB$ is a basis of $\topC$, then $\{\acP^{\,\inv}(B):B\in\basisB\}$ serves as a basis for $\acP^{\,\inv}(\topC)$. Equation \eqref{EQ basis top DP} and \eqref{EQ basis top WD} are now special cases of respectively \eqref{EQ basis top C} and \eqref{EQ basis top S}, where we note that the collection of all singletons forms a basis for the discrete topology.
\end{proof}

\subsection{Some Elementary Convergence Theorems}

\begin{theorem}\label{TH conv weakly discrete top}
Given $y\in S_0^{(<\infty)}$, we have that $\acP\,y^{[m]}$ converges to $\acP\,y$ in the discrete product topology $\topDP$. Further more, $y^{[m]}$ converges to the initial element $y\in S_0$ in the weakly discrete topology $\acP^{\,\inv}(\topDP)$.
\end{theorem}
\begin{proof}
Let $V$ be a neighborhood of $\acP\,y$ in the topology $\topDP$. By Lemma \ref{LM basis top}, without loss of generality we may assume that $V$ is an element of the form $\pi_n^{\inv}(c_0,\ldots,c_{n-1})$. By Remark \ref{RM impl theor}, we have that $\acP_i\,y^{[m]}=\acP_i\,y=c_i$ for all $i<m$, and so in particular for all $i<n$, when $m\geq n$. Hence $\acP\,y^{[m]}\in V$ for $m\geq n$, which proves by arbitrarity of $V$ that $\acP\,y^{[m]}$ converges to $\acP\,y$. The second claim follows by Remark \ref{RM topology}.
\end{proof}

Convergence in the weakly discrete topology turns out to be useful for our understanding of expansion systems, but its direct implications for the convergence of $y^{[n]}$ in other topologies are nonetheless limited. The following example will demonstrate this.

\begin{example}[\bf Fourier Series]\label{EX fourier approximation}

Let $S_i$ be the space of complex integrable functions on $[0,2\pi]$, such that $f(0)=f(2\pi)$. For $C_i$ we take the set $\C^2$. Now we define
$$P_i\,y :=(\int_0^{2\pi}e^{i\iota t}y(t)dt,\int_0^{2\pi}e^{-i\iota t}y(t)dt)=(c_{-i},c_{+i}),$$
where $\iota:=\sqrt{-1}$. Furthermore, we set
$$E_i\,y(x):=y(x)-c_{-i}e^{-i\iota x}-c_{+i}e^{i\iota x}.$$
One verifies that $\acP_i\,y=(c_{-i},c_{+i})$ corresponds with the $-i$-th and {$+i$-th} Fourier coefficient of $y$. As the $n-1$-th order Fourier approximation $y'_n:=c_{-(n-1)}e^{-(n-1)\iota x}+\ldots+c_{n-1}e^{(n-1)\iota x}$ is clearly an element of $S_0$, and we have that $\acE_n(y'_n)=0=\neut_n$, it follows by Theorem \ref{TH char prop} that $y^{[n]}$ equals the $n-1$-th order Fourier approximation.
\end{example}

Here, Theorem \ref{TH conv weakly discrete top} does not guarantee that $y^{[n]}$ will converge in other topologies than the weakly discrete topology. Indeed, in the example of Fourier approximation, a function may have well-defined Fourier coefficients, so that $y^{[n]}$ converges in the weakly discrete topology, while its Fourier series diverges in the topology of pointwise convergence (cf. \cite[Ch.~3, \S2.2]{StSh}). In fact, the weakly discrete topology is in this ES not even weaker than pointwise convergence, but just incomparable. For if we take the sequence $$f_n\colon [0,2\pi]\rightarrow \C,\quad f_n(x) = 1/n,$$
it clearly converges pointwise, but its zeroth Fourier coefficient does not converge in the discrete topology. However, given a topology we have in mind for $S_0$, it is often possible to turn the ES into a convergent one, simply by restricting the space $S_0$ to a smaller space. For example, in the case of Fourier approximation, in order to get pointwise convergence it suffices to restrict $S_0$ to differentiable functions (see \cite[Ch.~3, \S2.1]{StSh}).

Another important thing to notice, is the following. Theorem \ref{TH conv weakly discrete top} tells us that $y^{[n]}$ converges to $y$ in the weakly discrete topology. But, when we alter the values of the function $y$ on a set with Lebesgue measure $0$, then this will not have any effect on the Fourier coefficients $c_i$, and therefore also the convergents $y^{[n]}$ will not get altered. But according to Theorem \ref{TH conv weakly discrete top}, $y^{[n]}$ should also be accurately convergent towards the altered $y$ in the weakly discrete topology. This means that $\acP^{\,\inv}(\topDP)$ is not Hausdorff, and $y^{[n]}$ converges to multiple different elements at the same time. It turns out that the property described in the following definition, is characteristic for the weakly discrete topology to be Hausdorff.

\begin{definition}
An ES $\Psi$ is called {\em separating}\index{Expansion system!separating}, if the map $\acP\colon S_0\rightarrow \prodc$ is injective. In other words, $\Psi$ is called separating, if the family of projections $\acP_i$ separates points in $S_0$, i.e. for $y,y'\in S_0$ and $y\neq y'$, we have $\acP_i\,y\neq\acP_i\,y'$ for some $i\in \N$.
\end{definition}

So if $\Psi$ is a separating ES, the elements in $S_0$ can be uniquely codified as a list of coefficients $(c_0,c_1,\ldots)\in C$.

\begin{proposition}\label{PR hausdorff iff separating}
Let $\topC$ be a Hausdorff topology on $\prodc$, and let the element space $S_0$ be endowed with the generated topology $\acP^{\,\inv}(\topC)$. Then $S_0$ is Hausdorff iff $\Psi$ is separating.
\end{proposition}
\begin{proof}
If $\Psi$ is separating, then the map $\acP\colon S_0 \rightarrow \prodc$ is a injective continuous map into a Hausdorff space, which implies that $S_0$ is Hausdorff. On the other hand, if $\Psi$ is not separating, then we have $\seqc=\acP\,y=\acP\,y'\,\in C$ for certain $y,y'\in S_0, y\neq y'$. Now if $V_y$ and $V_{y'}$ are two neighborhoods of $y$ and $y'$ in $S_0$, we may w.l.o.g. assume that $V_y = \acP^{\,\inv}(U_y)$ and $V_{y'} = \acP^{\,\inv}(U_{y'})$ for some open neighborhoods $U_y$ and $U_{y'}$ in $\prodc$. Then $\seqc\in U_y\cap U_{y'}$, and so $\acP^{\,\inv}(c)\in V_y\cap V_{y'}$, which shows that any two neighborhoods of $y$ and $y'$ cannot be disjoint.
\end{proof}

\begin{theorem}\label{TH accurate if cont}
Suppose $\Psi$ is a separating ES, such that $S_0$ is a topological space and all $C_i$ are topological Hausdorff spaces. Moreover, suppose $\Psi$ is convergent w.r.t $y\in S_0$. If $\acP$ is continuous at $y^{[\infty]}$ in the weak topology, then $y = y^{[\infty]}$, which means $\Psi$ is accurate w.r.t. $y$. In particular, if $\Psi$ is a convergent ES and $\acP$ is continuous on $S_0$, then $\Psi$ is accurate.
\end{theorem}
\begin{proof}
If $\acP$ is continuous at $y^{[\infty]}$, we have that $\acP\,y^{[n]}$ converges to $\acP\,y^{[\infty]}$, and thus by Remark \ref{RM topology}, $y^{[n]}$ converges to $y^{[\infty]}$ in the weak topology. On the other hand $y^{[n]}$ converges to $y$ in the weakly discrete topology by Theorem \ref{TH conv weakly discrete top}, and thus also in the weak topology. As $S_0$ is Hausdorff in the weak topology by Proposition \ref{PR hausdorff iff separating}, we must have that $y = y^{[\infty]}$.
\end{proof}

\begin{remark}
Note that in the above theorem, for $\acP$ to be continuous at an element $y^{[\infty]}$, it is sufficient if all the expansions $E_i\colon S_i\rightarrow S_{i+1}$ and all the projections $P_i\colon S_i\rightarrow C_i$ are continuous at $y_i^{[\infty]}$, as each $\acP_i$ is composed from these maps. Also note that the theorem does not state that $y^{[n]}$ converges to $y$. It only says that if the sequence converges, and $\acP$ is continuous at its limit, then this limit can only be $y$.
\end{remark}

\begin{example}[\bf Convergence of Fourier Series]
We demonstrate how Theorem \ref{TH accurate if cont} can be applied to obtain accurateness of particular Fourier approximations. Consider the ES $\Psi$ from Example \ref{EX fourier approximation}, except that we let the $S_i$ consist of only continuous functions, with the topology of uniform convergence. The projections $\acP_i$ are then continuous on $S_0$, as we have the estimation:
\begin{eqnarray*}
\acP_i\,y - \acP_i\,y' = \acP_i(y-y') = \int_0^{2\pi} (y(t)-y'(t))e^{i\iota t}dt\\ \leq \int_0^{2\pi} |y(t)-y'(t)| dt \leq 2\pi\|y-y'\|.
\end{eqnarray*}
Using the property that Fourier coefficients are unique for continuous functions (cf. \cite[Cor.~2.2]{StSh}), we know that $\Psi$ is also separating. Now suppose $y\in S_0$ has absolutely convergent Fourier coefficients $c_i$, i.e.
$$
\sum_{i=-\infty}^\infty |c_i| < \infty.
$$
Then one easily verifies that the series:
$$
y^{[n]}(x) = \sum_{k=-n+1}^{n-1} c_{k}e^{k\iota x}
$$
converges uniformly on $[0,2 \pi]$. Hence by Theorem \ref{TH accurate if cont}, the approximations $y^{[n]}$ will converge to the initial element $y$.
\end{example}

\begin{lemma}\label{LM top indist}
Let $f\colon X\rightarrow Y$ be a mapping, and let $\topX$ be the topology on $X$. Then the following statements are equivalent:
\begin{enumerate}[label=\Roman{*}]
\item $\topX$ is weaker or equal to the generated topology $f^{\inv}(\topD)$, where $\topD$ is the discrete topology on $Y$.
\item Every open set $U$ of $\topX$ can be written as $f^{\inv}(V)$, where $V$ is a subset of $Y$.
\item $f(p) = f(p')$ implies that $p$ and $p'$ are topologically indistinguishable in $(X,\topX)$, i.e. every open set which contains one of these two elements, also contains the other.
\end{enumerate}
\end{lemma}
\begin{proof}
The equivalency between I and II is straightforward. So let's consider II implies III. Then, if $f(p)=f(p')$, and we have a open neighborhood $U$ containing $p$, it follows that $U=f^{\inv}(V)$ for some $V\subset Y$. But then $f(p)\in V$, so that $p'\in f^{\inv}(f(p))\subset f^{\inv}(V) = U$. By symmetry it follows that $p$ and $p'$ are indistinguishable.

Now we prove III implies II. Now let $U$ be an open set in $X$, and consider $V:=f(U)$. If $p\in f^{\inv}(V)$, than $f(p)=f(p')$ for some $p'\in U$, and so $p$ and $p'$ are topologically indistinguishable. Hence, we must have that $p\in U$ as well, which proves that $U\subset f^{\inv}(V)\subset U$. So $U$ can indeed be written as $f^{\inv}(V)$, with $V\subset Y$.
\end{proof}

This lemma, together with the following definition, allows us to formulate another convergence theorem.

\begin{definition}
Let $f\colon X\rightarrow Y$ be a map between topological spaces. We call such a map {\em closed at a point}\index{Map!closed at a point} $y\in Y$, if for every open neighborhood $U$ of the set $f^{\inv}(y)$, there is an open neighborhood $V$ of $y$, such that $f^{\inv}(V)\subset U$ (cf. \cite[4.5.13]{Enk}).
\end{definition}

\begin{remark}\label{RM closed iff closed at any point}
One can prove that a function $f\colon X\rightarrow Y$ is closed, iff $f$ is closed at any point in $Y$ according to the definition above (see for instance \cite[1.4.13]{Enk}).
\end{remark}

\begin{theorem}
Let $\Psi$ be a proper ES, and let $\prodc$ be endowed with the discrete product topology $\topDP$. Assume $S_0$ has a topology which is weaker or equal to $\acP^{\,\inv}(\topD)$. Then $\Psi$ is acurate w.r.t. $y$, if $\acP$ is closed at the point $\seqc = \acP(y)\in C$.
\end{theorem}
\begin{proof}
Let $U$ be a neighborhood of $y$. By Lemma \ref{LM top indist} we have that $U$ is also a neighborhood of $\acP^{\,\inv}(c)$. So by assumption, there is an open set $V\subset C$, such that $\acP^{\,\inv}(V)\subset U$, and by Lemma \ref{LM basis top} we may assume this $V$ is of the form $\pi_n^{\inv}(c_0,\ldots,c_{n-1})$, where $c_i=\acP_i\,y$. As $\acP_i\,y^{[m]}=c_i$ for $i<m$, we have that $y^{[m]}\in\acP^{\,\inv}(V)\subset U$ for all $m\geq n$, which proves convergence to $y$.
\end{proof}

\begin{corollary}
Let $\Psi$ be a proper ES, and let $\prodc$ be endowed with the discrete product topology $\topDP$. Then $\Psi$ is accurate, if $\acP$ is a closed map.
\end{corollary}
\begin{proof}
This follows immediately from Remark \ref{RM closed iff closed at any point}.
\end{proof}

\subsection{Convergence of Monotonic Expansion Systems}

In this subsection, we will define the so called {\em monotonic expansion systems}\index{Expansion system!monotonic}. This generalizes many instances of expansion systems, in which the element spaces are linearly ordered sets. The results we obtain for monotonic ES's can thus be applied to all the different instances of monotonic ES's.

Before we come to the definition of monotonic expansion systems, we will first briefly go through some definitions related to orders on a set. We call a function $f:X\rightarrow Y$ between linearly ordered sets $(X,\geq)$ and $(Y,\geq)$ {\em monotonic}\index{Map!monotonic}, if $f$ is either strictly increasing or strictly decreasing.\footnote{This meaning of the word monotonic might differ from those in other sources, where it could merely stand for strictly increasing or non-decreasing.}
Further, with the {\em dictionary order}\index{Order!dictionary order} on a (possibly infinite) Cartesian product $X=X_0\times X_1\times X_2\times\ldots$, we mean the order such that $(x_0,x_1,x_2,\ldots)<(x_0',x_1',x_2',\ldots)$, iff at the first entry $k$ on which the two sequences differ, we have that $x_k < x_k'$ (cf. \cite[Ch.1,\S3]{Mk}). The {\em order topology}\index{Topology!order topology} on an linearly ordered space $X$, can be defined as the topology generated by the subbasis of {\em open rays}\index{Topology!open ray} $X_{>x'}$ and $X_{<x''}$, with $x',x''\in X$. In particular this topology has as a basis the collection of {\em open intervals}\index{Topology!open interval} $X_{x'<\cdot<x''}$ with $x',x''\in X$, together with the sets $X_{\minelt{x}\,\leq\cdot<x''}$ and $X_{x'<\cdot\leq \maxelt{x}}$ in case $X$ contains a minimum element $\minelt{x}$ or a maximum element $\maxelt{x}$ (cf. \cite[Ch.2,\S14]{Mk},\cite[Ch.~1, \S2.1]{ArPo}).\footnote{To be fully correct, we should remark that a basis must also contain the set $X_{\minelt{x}\,\leq\cdot\leq \maxelt{x}}$, in the rare case that $X$ only contains the single point $\minelt{x}=\maxelt{x}$.} A subset $U$ of an ordered set $X$ is said to be {\em densely ordered}\index{Order!densely ordered}, if for any two different elements $x,x'\in X$ there lies an element $u\in U$ in between.

\begin{definition}
Suppose we have an expansion system $\Psi$, for which the sequences $S$ and $C$ both consist of only linearly ordered sets. Then, if each $F_i:S_i\rightarrow C_i\times S_{i+1}$ is monotonic when $C_i\times S_{i+1}$ is given the dictionary order, we call $\Psi$ a monotonic expansion system. More concrete, this condition on $F_i$ means that either for all elements $y<y'$ in $S_i$, we have that:
\begin{align}
P_i\,y&\leq P_i\,y'\quad\text{and}\label{EQ P increasing}\\
E_i\,y&<E_i\,y'\quad\text{whenever}\ P_i\,y=P_i\,y'.\label{EQ E increasing}
\end{align}
or for all elements $y<y'$ we have that:
\begin{align}
P_i\,y&\geq P_i\,y'\quad\text{and}\label{EQ P decreasing}\\
E_i\,y&>E_i\,y'\quad\text{whenever}\ P_i\,y=P_i\,y'.\label{EQ E decreasing}
\end{align}
If all $F_i$ are strictly increasing (i.e. we have \eqref{EQ P increasing} and \eqref{EQ E increasing}), we call $\Psi$ a {\em strictly increasing expansion system}.\index{Expansion system!strictly increasing}
\end{definition}

\begin{example}\label{EX first ex mon es}
The decimal expansion (Example \ref{EX dec exp 2}), is a strictly increasing ES. One easily verifies that \eqref{EQ P increasing} and \eqref{EQ E increasing} both hold. The continued fraction (Example \ref{EX cont frac}) is not a strictly increasing ES, but it is a monotonic ES, for we have that \eqref{EQ P decreasing} and \eqref{EQ E decreasing} holds for all $i\in \N$, where we should remark that $\infty$ is to be interpreted as the largest element in $C_i$.
\end{example}

From the definition of a strictly increasing ES, we almost immediately obtain the following property:

\begin{proposition}\label{PR non-dec}
Let $\Psi$ be strictly increasing expansion system, and let $\prodc$ be given the dictionary order. Then $\acP:S_0\rightarrow \prodc$ is a non-decreasing map.
\end{proposition}
\begin{proof}
Suppose we have $\acP\,y>\acP\,y'$ for certain $y<y'$. Then there must be a $k\geq 0$, such that $P_i\,y=P_i\,y'$ for $i<k$, and $P_k\,y> P_k\,y'$. But this contradicts \eqref{EQ P increasing} for $i=k$, from which we conclude that $\acP$ is non-decreasing.
\end{proof}

This simple proposition will be useful in the proof of the following two important theorems.

\begin{theorem}\label{TH densely ord}
Suppose $\Psi$ is a monotonic expansion system. If the set of elements of finite order $S^{[<\infty]}$ is densely ordered in $S_0$, then $\Psi$ is separating.
\end{theorem}
\begin{proof}
First we prove the theorem for strictly increasing expansion systems, and then we will show how the result can be naturally extended to monotonic ones. So let us be given a strictly increasing ES $\Psi$, and give $\prodc$ the dictionary order. If $\Psi$ would not be separating, then there are two different elements $y',y''\in S_0$, such that $\acP\,y'=\acP\,y''$. By Proposition \ref{PR non-dec}, we know that $\acP$ is non-decreasing, and hence $\acP y = \acP y'$ for all $y'\leq y \leq y''$. In particular, we can pick two different elements $z$,$z'$ of finite order, lying both between $y'$ and $y''$ (we can first apply the property of being densely ordered on $y'$ and $y''$ to get $z$, and then on $z$ and $y''$ to get $z'$). But then $z$ and $z'$ must have exactly the same coefficients, which contradicts that an element of finite order is uniquely determined by its coefficients (see Remark \ref{RM impl theor}). Hence $\Psi$ must be separating.

Now to show that the same is true if all $\acP_i$ are only monotonic. This can be achieved by transforming our monotonic $\Psi$ into a strictly increasing expansion system $\Psi'$, by reversing the order relations on some of the spaces $S_i$ and $C_i$. Note that reversing an order on a linearly ordered space keeps the space linearly ordered. The reversing procedure will be according to the following rule: starting from $i=0,1,\ldots$, we reverse the order on $S_{i+1}$ and $C_{i+1}$ either if $F_i$ is a strictly decreasing map in $\Psi$ and we haven't reversed the order of $S_i$ and $C_i$ in $\Psi'$, or if $F_i$ is strictly increasing in $\Psi$, but we also reversed the order on $S_i$ and $C_i$ in $\Psi'$. By doing this, we keep al strictly increasing maps strictly increasing, and at the same time strictly decreasing maps in $\Psi$ become strictly increasing ones in $\Psi'$, so that our new $\Psi'$ is a strictly increasing ES. As reversing orderings has neither any effect on the element and coefficient spaces themselves, nor on the collection of expansions and projections (in fact we are only changing properties which lie outside the essential structure of the ES), the claim of the theorem also applies to $\Psi'$.\footnote{Also the property of $S^{[<\infty]}$ being densely ordered in $S_0$ cannot have been affected by the process of reversing orderings, even if we would have reversed the order on $S_0$ itself, which we haven't.} Hence, the statement extends without problems to monotonic expansion systems.
\end{proof}

\begin{theorem}\label{TH monotonic es}
Suppose $\Psi$ is a proper monotonic expansion system, and $S_0$ has the order topology. Then $\Psi$ is accurate, iff $\Psi$ is separating.
\end{theorem}
\begin{proof}

As in Theorem \ref{TH densely ord}, we first prove the statement for strictly increasing $\Psi$, and then extend it to the case $\Psi$ is monotonic.

First let us assume that $\Psi$ is accurate, but not separating. Then there are $y,y'\in S_0$, having equal coefficients and $\Psi$ is accurate w.r.t. both of them. This means that $y^{[n]}=y'^{[n]}$ converges to two elements at the same time, and so $S_0$ cannot be a Hausdorff space. But $S_0$ has the order topology, which is clearly an Hausdorff space. Hence $\Psi$ must be separating.

To prove the reverse implication, first we give $\prodc$ the dictionary order, as well as the order topology $\topC$ generated by this ordering. We now prove that the discrete product topology $\topDP$, is stronger than the dictionary order topology $\topC$ on $\prodc$. For let $V$ be an open set in the dictionary order topology, containing an element $\seqc=(c_0,c_1,\ldots)$. W.l.o.g. we may assume that $V$ is either $\prodc_{\seqc'<\cdot<\seqc''}$ or $\prodc_{\minelt{\seqc}\,\leq\cdot<\seqc''}$ or $\prodc_{\seqc'<\cdot\leq \maxelt{\seqc}}$, where $\minelt{\seqc}$ and $\maxelt{\seqc}$ denote respectively the smallest and the largest element in $\prodc$, if it exists. First we regard the case $\prodc_{\seqc'<\cdot<\seqc''}$. As $\seqc$ is supposed to lie between $\seqc'$ and $\seqc''$, we have on the one hand that $\seqc'$ and $\seqc$ have equal coefficients up to and including $n'-1$ and $c_{n'}'<c_{n'}$. On the other hand we have that $\seqc$ and $\seqc''$ have equal coefficients up to and including $n''-1$, and $c_{n''}''<c_{n''}$. This means that any coefficient sequence with the same coefficients as $\seqc$ up to and including $m=\max(n',n'')$ will lie in the set $\prodc_{\seqc'<\cdot<\seqc''}$. Hence $\pi_n^{\inv}(c_0,\ldots,c_{m})$ is a neighborhood of $\seqc$ in the discrete product topology, that is contained in $V$.

If $V$ is a set of the form $\prodc_{\minelt{\seqc}\,\leq\cdot<\seqc''}$, then we may assume that $\seqc$ is the minimal element $\minelt{\seqc}$, for otherwise we can apply the previous result to $\prodc_{\minelt{\seqc}<\cdot<\seqc''}$. But in that case, we can consider the set $\pi_n^{\inv}(c_0,\ldots,c_{m})$, where $m$ is the first coefficient such that $c_m<c_m''$. A similar argument applies to the case $V=\prodc_{\seqc'<\cdot\leq \maxelt{\seqc}}$, and thus it follows that the discrete product topology is stronger or equal to the dictionary order topology.

By Theorem \ref{TH conv weakly discrete top}, we have that $\acP\,y^{[n]}$ converges to $\acP\,y$ in the discrete product topology. Therefore, $\acP\,y^{[n]}$ also converges to $\acP\,y$ in the dictionary order topology $\topC$. This, in turn, implies that $y^{[n]}$ converges to $y$ in the generated topology $\acP^{\inv}(\topC)$. By Proposition \ref{PR non-dec}, we know that $\acP$ is non-decreasing, and as $\Psi$ is supposed to be separating, so that $\acP$ is injective, we conclude that it must be strictly increasing.  Therefore, we have that the topology $\acP^{\inv}(\topC)$ is stronger or equal to the order topology on $S_0$, for we can write $(S_0)_{>y'}$ as $\acP^{\inv}(\prodc_{>\acP(y')})$, and likewise for $(S_0)_{<y''}$. Hence, we have at last that $y^{[n]}$ converges to $y$ in the order topology of $S_0$, which means that $\Psi$ is accurate.

To show that the same is true if all $\acP_i$ are only monotonic, we apply exactly the same trick as we have done in the proof of Theorem \ref{TH densely ord}. Again, this doesn't affect the essential properties of $\Psi$, nor does it affect the order topology on $S_0$. Hence, the result follows for proper monotonic ES's as well.
\end{proof}

\begin{corollary}\label{CR dens ord acc}
Suppose $\Psi$ is a proper monotonic expansion system, and $S_0$ has the order topology. If $S^{[<\infty]}$ is densely ordered in $S_0$, then $\Psi$ is accurate.
\end{corollary}
\begin{proof}
This follows immediately by combining Theorem \ref{TH densely ord} with Theorem \ref{TH monotonic es}.
\end{proof}

This corollary can be used to prove convergence of the ES's in Example \ref{EX first ex mon es}:

\begin{example}\label{EX acc dec exp cont frac}
The decimal expansion (Example \ref{EX dec exp 2}) and the continued fraction (Example \ref{EX cont frac}) are accurate, if $S_0$ is given the standard topology (i.e. the Euclidean topology) on the real numbers.
\end{example}
\begin{proof}
That these expansion systems are proper and monotonic, we have already shown in former examples. Noting that the standard topology on the reals coincides with the order topology (cf. \cite[\S14, Ex.1]{Mk}), it remains to prove that in both cases $S^{[<\infty]}$ is densely ordered in $S_0$. For the decimal expansion, this is easy, as we have that the elements of finite order are exactly all multiples of $10^{-k}$ between (and including) $0$ and $1$, where $k$ can be any positive number. These numbers are easily seen to be densely ordered (given two of them, we can multiply by a suitable power of $10$, so that it boils down to picking a natural number between two other natural numbers). Hence $\Psi$ is accurate in the case of decimal expansion.

In case of the continued fractions, we claim that the elements of finite order are exactly the fractions between (and including) $0$ and $1$. To see this, we can look at the function $M:\Q_{0\leq\cdot<1}\rightarrow \N$, which assigns to a fraction $p/q$ (where $p$ and $q$ are written in their lowest terms) the value $p+q$. Then one verifies that for $(p,q)\neq (0,1)$:
$$
M\left(E_i \left(\frac{p}{q}\right)\right) = M\left(\frac{q-c_i p}{p}\right)= q-(c_i-1)p<p+q=M\left(\frac{p}{q}\right)
$$
So if $p/q$ is non-zero, then the value of $M$ decreases with each application of $E_i$, while it should remain positive. This can only be the case if for each fraction $y$ in $S_0$, $\acE_n\,y$ equals $0 = \frac{0}{1}$ for $n$ big enough, and this $n$ is by definition the order of $y$. Also the rational numbers are of course densely ordered, and so by Corollary \ref{CR dens ord acc}, we have that the ES of continued fractions is accurate.
\end{proof}

There are many more instances of monotonic expansion systems. We will give some more examples.

\begin{example}[\bf Egyptian Fractions and Engel Expansion]\label{EX eg frac en exp}
Let $(S_i,\neut_i)=([0,1),0)$, and $C_i=\Z_{\geq 1}\cup\{\infty\}$. Furthermore, let $P_i$ be given by
$$
P_i\,y = \begin{cases} \infty &\text{if}\quad y = 0\\
\lceil 1/y\rceil, &\text{otherwise}, \end{cases}
$$
and $E_i$ by
$$
E_i\,y = \begin{cases} 0 &\text{if}\quad y = 0\\
y-\frac{1}{P_i\,y} &\text{otherwise}. \end{cases}
$$
Now $F_i = (P_i,E_i)$ is a bijection, and we have as inverse:
$$
F_i^{\inv}(c,y)= \begin{cases} 0 &\text{if}\quad c = \infty\\
\frac{1}{c}+y &\text{otherwise}. \end{cases}
$$
By Proposition \ref{PR bijective F}, this $\Psi$ is proper. If we give each $C_i=\Z_{\geq 1}\cup\{\infty\}$ the reversed order, then each $F_i$ is a strictly increasing map. In the way we have demonstrated in Example \ref{EX acc dec exp cont frac}, we can prove that $S_0^{[<\infty]}$ is again exactly the set of rational numbers in $S_0$ (this time, look at $M(p/q)=p$), so that $\Psi$ is accurate by Corollary \ref{CR dens ord acc}. One easily sees that the approximation this ES provides, is just the sum of reciprocals of the coefficients. For example, if we take the element $y = 1/\sqrt{2}\in S_0$, then we get the approximation:
$$
\frac{1}{\sqrt{2}} = \frac1 2 + \frac1 5 + \frac1 {141} + \frac1 {68575} + \ldots.
$$
Breaking off the series at a given depth, the resulting approximation will be an expression which is a sum of only unit fractions. Such expressions are usually referred to as {\em Egyptian fractions}.

Now let's make a slight change to our $\Psi$: we keep all components the same, except for $E_i$, for which we set:
$$
E_i\,y = \begin{cases} 0 &\text{if}\quad y = 0\\
y\cdot P_i(y)-1 &\text{otherwise}. \end{cases}
$$
The inverse of $F_i$ is then given by:
$$
F_i^{\inv}(c,y) = \begin{cases} 0 \quad&\text{if}\quad c = \infty\\
\frac{1+y}{c} \quad&\text{otherwise}. \end{cases}
$$
Again, $\Psi$ is proper by Proposition \ref{PR bijective F}, and each $F_i$ is strictly increasing if we give $C_i$ the reversed order. $S^{[<\infty]}$ is the set of rationals (look at $M(p/q)=p$), and thus $\Psi$ is accurate by Corollary \ref{CR dens ord acc}. For the $n$-th approximation we find:
\begin{align*}
y_n^{[n]} &= 0\\
y_{n-1}^{[n]} &= \frac{1}{c_{n-1}}\\
y_{n-2}^{[n]} &= \cfrac{1+\cfrac{1}{c_{n-1}}}{c_{n-2}}\\
\ldots\\
y_0^{[n]} &= \cfrac{1+\cfrac{1+\cfrac{1+\cfrac{1}{c_{n-1}}}{\ldots}}{c_1}}{c_0}.
\end{align*}
Hence we have the concise expression for $y$:
\begin{equation}\label{EQ eng exp}
y\ =\ \cfrac{1+\cfrac{1+\cfrac{1+\ldots}{c_2}}{c_1}}{c_0}\ =\ \frac1 {c_0}+\frac1 {c_0c_1}+\frac1 {c_0c_1c_2} + \ldots
\end{equation}
Here, the approximations are a special kind of Egyptian fractions, in which the denominators of consecutive fractions increase by a non-decreasing sequence of factors $c_0,c_1,\ldots$. Such approximations are called {\em Engel expansions} (cf. \cite{En}\cite{ErReSz}).
\end{example}

\begin{example}[\bf Taylor Series]\label{EX tay ser mon}
Let $\Psi$ be as in Example \ref{EX taylor}, except that we confine the spaces $S_i$ to only real holomorphic functions. We can view $\Psi$ now as a strictly increasing ES, if we give $S_i$ the ordering defined by:
\begin{equation}\label{EQ ord tay}
y<y'\quad\text{if}\quad y(x)<y'(x)\quad\text{for all}\ x\in (x_0,x_0+\epsilon),
\end{equation}
for some sufficiently small $\epsilon>0$, depending on $y$ and $y'$. By the limit properties of $y$ in $x_0$, this ordering coincides with the dictionary order on the derivatives of $y$, and thus the Taylor coefficients. So this demonstrates that $\Psi$ is monotonic. As for the accurateness of $\Psi$: this is not a very interesting question, as Taylor approximations will at least locally converge just by the definition of holomorphic functions. But indeed, the elements of finite order are the polynomials, and in the ordering we have just put on our element spaces, these are densely ordered in the set of holomorphic functions. So by Corollary \ref{CR dens ord acc}, the Taylor polynomials converge to the initial function in the ordering given by \eqref{EQ ord tay}.
\end{example}

The monotonic expansion systems seem to capture quite some well-known expansions. As for the continued fraction, in which the expanding maps $E_i$ are all strictly decreasing, a similar generalization has priorly been studied by Bissinger in \cite{Bi}, and for the decimal expansions, in which the expanding maps are strictly increasing, a generalization has been introduced by Everett in \cite{Ev}. Convergence and other properties have been further investigated in \cite{Me}\cite{Pa}\cite{Rec}\cite{Ren}. As R\'enyi remarked in \cite{Ren}, the formulation of Bissinger and Everett are somewhat stricter than necessary. In an elementary formulation (additional requirements could be imposed to obtain convergence), the so-called "f-expansion" of Bissinger and Everett can be translated into an expansion system as follows:

\begin{example}[\bf f-Expansion]\label{EX f expansion}
Let $\Psi$ be such that each $(S_i,\neut_i)$ equals $([0,1),0)$, and let $f:[0,1)\rightarrow I$ be a monotonic mapping from $[0,1)$ onto a certain interval $I$ of the extended real line $\R\cup\{-\infty,+\infty\}$. Let each $C_i$ equal $\Z\cup\{-\infty,+\infty\}$. We set all $P_i$ identical to
$$
P_i\,y = \lfloor f(y) \rfloor,
$$
and all $E_i$ identical to
$$
E_i\,y = f(y) - P_i\,y = f(y) - \lfloor f(y) \rfloor,
$$
where $\lfloor -\infty \rfloor:= -\infty$ and $\lfloor +\infty \rfloor:= +\infty$. It should be noted that, because $[0,1)$ is an half open interval, the value $+\infty$ can only be attained for strictly decreasing $f$, and $-\infty$ for strictly increasing $f$. This expansion system models Everett's $f$-expansion for strictly increasing $f$, and Bissinger's $f$-expansion for strictly decreasing $f$.
\end{example}

\newpage
\section{Isomorphisms between Expansion Systems}

In this section we will introduce the concept of homomorphisms, and more importantly isomorphisms between ES's. When two ES's are isomorphic, then they essentially have the same underlying structure, and so they also share the other properties that depend on this structure. The isomorphism itself serves as a translation from one ES to another. Definition \ref{DF homom isom} will give us the tools to construct such homomorphisms and isomorphisms.

\subsection{Homomorphisms and Isomorphisms}

\begin{definition}\label{DF homom isom}
Let $\Psi=(S,C,F)$ and $\alt \Psi = (\alt S,\alt C,\alt F)$ be two expansion systems. A pair of maps $\Lambda = (\LambdaE{S},\LambdaE{C}) = (\{\LambdaU{S}{i}\}_{i\in \N},\{\LambdaU{C}{i}\}_{i\in \N})$, such that \begin{align}
\LambdaU{S}{i}\colon  S_i &\rightarrow \alt S_i\label{EQ lambda S}\\
\LambdaU{C}{i}\colon  C_i &\rightarrow \alt C_i\label{EQ lambda C}\\
\intertext{is called a {\em homomorphism}\index{Expansion system!homomorphism} of expansion systems, if the following equations hold:}
\LambdaU{S}{i} \neut_i &= \alt \neut_i\label{EQ hom neut}\\
\LambdaU{S}{i+1} E_i &= \alt E_i \LambdaU{S}{i}\label{EQ hom E}\\
\LambdaU{C}{i} P_i &= \alt P_i \LambdaU{S}{i}.\label{EQ hom P}
\end{align}

If this $\Lambda$ consists of only injective maps, then $\Lambda$ is called an {\em embedding} of $\Psi$ in $\alt \Psi$\index{Expansion system!embedding}. If $\Lambda$ is bijective, then we call it an {\em isomorphism}\index{Expansion system!isomorphism} and we say that $\Psi$ and $\alt \Psi$ are isomorphic.
\end{definition}

Based on this definition, we can already get an idea of its purpose by the following proposition.

\begin{proposition}\label{PR homom}
Suppose $\Psi$ and $\alt \Psi$ are ES's, and $\Lambda$ is an homomorphism from $\Psi$ into $\alt \Psi$. If $y\in S_0^{(n)}$, then $\LambdaU{S}{0}\,y\in  \alt{S_0}^{(n)}$, and we have $(\LambdaU{S}{0}\, y)^{[\Psi',n]}=\LambdaU{S}{0}\,y^{[\Psi,n]}$.
\end{proposition}
\begin{proof}
If $y\in S_0^{(n)}$, it follows by Theorem \ref{TH char prop} that there is an element $y^{[\Psi,n]}\in S_0^{[n]}$ with the same coefficients as $y$ up to and including order $n-1$. Now by \eqref{EQ hom neut} and \eqref{EQ hom E}, we have that $\acE_n(\LambdaU{S}{0}\,y^{[\Psi,n]}) = \alt \neut_n$ and thus $\LambdaU{S}{0}\,y^{[\Psi,n]}\in \alt{S_0}^{[n]}$. By \eqref{EQ hom P} we have that $\LambdaU{S}{0}\,y^{[\Psi,n]}$ has the same coefficients as $\LambdaU{S}{0}\,y$ up to and including order $n-1$, and thus Theorem \ref{TH char prop} tells us that $\LambdaU{S}{0}\,y\in \alt{S_0}^{(n)}$ and $(\LambdaU{S}{0} \,y)^{[\Psi',n]}=\LambdaU{S}{0}\,y^{[\Psi,n]}$.
\end{proof}

The equations \eqref{EQ hom E} and \eqref{EQ hom P} can be rephrased in terms of diagrams. These requirements are satisfied, iff the diagrams in Figure \ref{FI hom E hom P} commute.
\begin{figure}[h]
\centering
\begin{minipage}[b]{0.35\linewidth}
\centering
 \begin{tikzpicture}[descr/.style={fill=white,inner sep=2.5pt}]
 \matrix (m) [matrix of math nodes, row sep=3em,
 column sep=3em]
 { S_i & S_{i+1}\\
   \alt S_i & \alt S_{i+1}\\ };
 \path[->,font=\scriptsize]
 (m-1-1) edge node[auto] {$ E_i $} (m-1-2)
         edge node[auto] {$ \LambdaU{S}{i} $} (m-2-1)
 (m-1-2) edge node[auto] {$ \LambdaU{S}{i+1} $} (m-2-2)
 (m-2-1) edge node[auto] {$ \alt E_i $} (m-2-2);
 \end{tikzpicture}
\end{minipage}
\hspace{0.0cm}
\begin{minipage}[b]{0.35\linewidth}
\centering
 \begin{tikzpicture}[descr/.style={fill=white,inner sep=2.5pt}]
 \matrix (m) [matrix of math nodes, row sep=3em,
 column sep=3em]
 { S_i & C_{i} \\
   \alt S_i & \alt C_{i}\\ };
 \path[->,font=\scriptsize]
 (m-1-1) edge node[auto] {$ P_i $} (m-1-2)
         edge node[auto] {$ \LambdaU{S}{i} $} (m-2-1)
 (m-1-2) edge node[auto] {$ \LambdaU{C}{i} $} (m-2-2)
 (m-2-1) edge node[auto] {$ \alt P_i $} (m-2-2);
 \end{tikzpicture}
\end{minipage}
\caption{Equations \eqref{EQ hom E} and \eqref{EQ hom P}.}
\label{FI hom E hom P}
\end{figure}
From these diagrams it easily follows that if $\Lambda$ is bijective, then the set of inverse mappings $\Lambda^{\inv} = ((\LambdaE{S})^{\inv},(\LambdaE{C})^{\inv})$ is an isomorphism as well. This gives us the following corollary to Proposition \ref{PR homom}.

\begin{corollary}\label{CR isom}
If $\Psi$ and $\alt \Psi$ are isomorphic by the isomorphism $\Lambda$, then $y\in S_0^{(n)}$ iff $\LambdaU{S}{0}y\in \alt{S_0}^{(n)}$, and we have:
\begin{align}
y^{[\Psi,n]}&=(\LambdaU{S}{0})^{\inv}(\LambdaU{S}{0}\,y)^{[\Psi',n]},\text{\ and}\\ {y'}^{[\Psi',n]}&=\LambdaU{S}{0}\,((\LambdaU{S}{0})^{\inv}\,y')^{[\Psi,n]}.
\end{align}
\end{corollary}

\begin{remark}\label{RM isom prop}
In particular, this corollary says that if $\Psi$ and $\Psi'$ are isomorphic, then properness of one of them also guarantees properness of the other. This cannot be said of accurateness, as the concept of isomorphy is indifferent to which topology is given on the initial element space $S_0$. However, if we have that $\LambdaU{S}{0}$ is a homeomorphism between $S_0$ and $S_0'$, then of course convergence or accurateness of one of the ES's, implies respectively convergence and accurateness of the other.
\end{remark}

Furthermore, it should be noted that the relationship of being isomorphic defines an equivalence relation among expansion systems, as is easily verified.

We will now regard some elementary properties for commutative diagrams, which we will then apply to the situation above. These properties can be best illustrated with the diagrams of Figure \ref{FI prop com diag}.
\begin{figure}[h]
\centering
\subfloat[]{\begin{minipage}[t]{0.25\linewidth}

\label{FI inj exists}
\begin{tikzpicture}[descr/.style={fill=white,inner sep=2.5pt}]
 \matrix (m) [matrix of math nodes, row sep=3em, column sep=3em,ampersand replacement=\&]
 {X \& Y' \\
  Y \& Z \\};
 \path[->,font=\scriptsize]
    (m-1-1) edge (m-1-2)
    (m-1-2) edge (m-2-2);
 \path[right hook->,font=\scriptsize]
    (m-1-1) edge node[auto] {$ i $} (m-2-1);
 \path[->, loosely dashed, font=\scriptsize]
    (m-2-1) edge node[auto] {$ \phi_a $} (m-2-2)
    (m-2-1) edge node[below, sloped] {exists} (m-2-2);
\end{tikzpicture}
\end{minipage}}
\subfloat[]{\begin{minipage}[t]{0.25\linewidth}

\label{FI surj exists}
\begin{tikzpicture}[descr/.style={fill=white,inner sep=2.5pt}]
 \matrix (m) [matrix of math nodes, row sep=3em, column sep=3em,ampersand replacement=\&]
 {X \& Y' \\
  Y \& Z \\};
 \path[->,font=\scriptsize]
    (m-1-1) edge (m-1-2)
    (m-1-2) edge (m-2-2);
 \path[->>,font=\scriptsize]
    (m-2-1) edge node[auto] {$ s $} (m-2-2);
 \path[->, loosely dashed, font=\scriptsize]
    (m-1-1) edge node[auto] {$ \phi_b $} (m-2-1)
    (m-1-1) edge node[below, sloped] {exists} (m-2-1);
\end{tikzpicture}
\end{minipage}}
\subfloat[]{\begin{minipage}[t]{0.25\linewidth}

\label{FI surj unique}
\begin{tikzpicture}[descr/.style={fill=white,inner sep=2.5pt}]
 \matrix (m) [matrix of math nodes, row sep=3em, column sep=3em,ampersand replacement=\&]
 {X \& Y' \\
  Y \& Z \\};
 \path[->,font=\scriptsize]
    (m-1-1) edge (m-1-2)
    (m-1-2) edge (m-2-2);
 \path[->>,font=\scriptsize]
    (m-1-1) edge node[auto] {$ s $} (m-2-1);
 \path[->, densely dotted, font=\scriptsize]
    (m-2-1) edge node[auto] {$ \phi_c $} (m-2-2)
    (m-2-1) edge node[below, sloped] {unique} (m-2-2);
\end{tikzpicture}
\end{minipage}}
\subfloat[]{\begin{minipage}[t]{0.25\linewidth}

\label{FI inj unique}
\begin{tikzpicture}[descr/.style={fill=white,inner sep=2.5pt}]
 \matrix (m) [matrix of math nodes, row sep=3em, column sep=3em,ampersand replacement=\&]
 {X \& Y' \\
  Y \& Z \\};
 \path[->,font=\scriptsize]
    (m-1-1) edge (m-1-2)
    (m-1-2) edge (m-2-2);
 \path[right hook->,font=\scriptsize]
    (m-2-1) edge node[auto] {$ i $} (m-2-2);
 \path[->, densely dotted, font=\scriptsize]
    (m-1-1) edge node[auto] {$ \phi_d $} (m-2-1)
    (m-1-1) edge node[below, sloped] {unique} (m-2-1);
\end{tikzpicture}
\end{minipage}}
\caption{Four Properties of Commutative Diagrams.}
\label{FI prop com diag}
\end{figure}
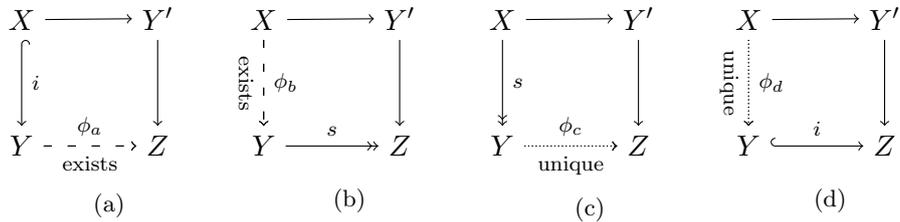
The first of these four diagrams, Diagram \subref{FI inj exists}, says that if we have only three of the four maps, and there is an injective map $i$ leading to the map that is missing, then a fourth map $\phi_a$ must exist, which completes the commutative diagram. The same can be said if there is a surjective map $s$, such that the map that is missing would lead to the map $s$, as is shown in Diagram \subref{FI surj exists}. Diagram \subref{FI surj unique} expresses that if we have three out of four maps, and a surjective map $s$ leads to the map missing, then any map $\phi_c$ completing the diagram has to be unique, i.e. there is at most one map $\phi_c$ which could complete the diagram. Finally, Diagram \subref{FI inj unique} says the same thing for the case that the missing map $\phi_d$ would lead to some injective map $i$. One may verify that maps are a solution for these diagrams, if and only if we have:
\begin{align}
\phi_a|_{i(X)}\colon i(x) &\mapsto f(x),\label{EQ phi a}\\
\phi_b(x)&\in s^{-1}(f(x)),\label{EQ phi b}\\
\phi_c\colon s(x) &\mapsto f(x),\label{EQ phi c}\\
\phi_d\colon x &\mapsto i_l^{-1}\circ f(x),\label{EQ phi d}
\end{align}
where the map $f:X\rightarrow Z$ denotes the composed map from $X$ to $Z$ via $Y'$, and $i_l^{-1}$ is the left inverse of $i$ on $f(X)$. The map $\phi_d$ in \eqref{EQ phi d} can be defined if and only if $f(X)\subset i(Y)$. For our situation of expansion systems, the above statements have, among others, the following implications:

\begin{theorem}\label{TH diag exists unique}
Let us be given two sequences of element spaces $S$ and $S'$, and two sequences of coefficient spaces $C$ and $C'$. Throughout the following statements, we do not make any prior assumptions about existence of the connecting sequences $E$,$E'$,$P$,$P'$,$\LambdaE{S}$ or $\LambdaE{C}$, except for those which are mentioned explicitly. Then we have that:
\begin{enumerate}
\item[A1] Given $E$ and injective $\LambdaE{S}$, then an $E'$ satisfying \eqref{EQ hom E} exists.
\item[A2] Given $P$,$\LambdaE{C}$ and injective $\LambdaE{S}$, then a $P'$ satisfying \eqref{EQ hom P} exists.
\item[B] Given $E'$ and surjective $\LambdaE{S}$, then an $E$ satisfying \eqref{EQ hom E} exists.
\item[C1] Given $E$ and surjective $\LambdaE{S}$, then the unique map $E_i'$ satisfying \eqref{EQ hom E} exists, iff it sends $\LambdaU{S}{i}\,y$ to $\LambdaU{S}{i+1}E_i\,y$.
\item[C2]
Given $P$,$\LambdaE{C}$ and surjective $\LambdaE{S}$, then the unique map $P'$ satisfying \eqref{EQ hom P} exists, iff it sends $\LambdaU{S}{i}\,y$ to $\LambdaU{C}{i}P_i\,y$.
\item[D] Given $E'$ and injective $\LambdaE{S}$, then the unique map $E$ satisfying \eqref{EQ hom E} exists, iff $E_i'\LambdaU{S}{i}(S_i)\subset \LambdaU{S}{i+1}(S_{i+1})$.
\end{enumerate}
\end{theorem}

\begin{remark}
Different from the other statements of Theorem \ref{TH diag exists unique}, the claims B and D can be used in case the ES $\Psi'$ is given, and one wants to embed a certain $\Psi=(S,C,F)$ into this $\Psi'$.
w\end{remark}

\begin{theorem}\label{TH isom es}
Let $\Psi$ be an ES, and let $S'$ and $C'$ be two other sequences of respectively element and coefficient spaces. Then there exists an ES $\Psi'=(S',C',F')$ isomorphic to $\Psi$, iff there is a collection of bijective mappings $\LambdaE{C}$ between $C$ and $C'$, and bijective mappings $\LambdaE{S}$ between $S$ and $S'$, satisfying \eqref{EQ hom neut}. Given such bijective $\LambdaE{S}$ and $\LambdaE{C}$, $F'$ is uniquely determined by:
\begin{align}
E_i' &= \LambdaU{S}{i+1}\,E_i\,(\LambdaU{S}{i})^{\inv}\quad\text{and}\label{EQ uniq E}\\
P_i' &= \LambdaU{C}{i}\,P_i\,(\LambdaU{S}{i})^{\inv}.\label{EQ uniq P}
\end{align}
\end{theorem}
\begin{proof}
The existence of a sequence $F'=(P',E')$ satisfying \eqref{EQ hom E} and \eqref{EQ hom P} follows from Theorem \ref{TH diag exists unique}A1 and A2. As $\LambdaE{S}$ is bijective, Theorem \ref{TH diag exists unique}C1 and C2 tells us that $E'$ and $P'$ are uniquely determined by \eqref{EQ uniq E} and \eqref{EQ uniq P}.
\end{proof}

\begin{corollary}\label{CR autom isom}
Let $\Psi=(S,C,F)$ be an ES. Then an ES $\Psi'=(S,C,F')$ is isomorphic to $\Psi$, iff there is a collection of automorphisms $\LambdaE{C}$ on $C$, and automorphisms $\LambdaE{S}$ on $S$, satisfying \eqref{EQ hom neut}, such that $F'$ can be defined as in \eqref{EQ uniq E} and \eqref{EQ uniq P}.
\end{corollary}

\begin{theorem}\label{TH isom shift}
Let $\Psi$ be an ES, and suppose each map $E_i$ is a composition of two maps $E_i^{(1)}\colon S_i\rightarrow S_i'$ and $E_i^{(2)}\colon S_i'\rightarrow S_{i+1}$, $E_i^{(1)}$ being bijective. Let $\Psi'=(S',C',F')$ be a another ES, where $S_i'$ is the domain of $E_i^{(2)}$, and which is further defined as:
\begin{align}
\neut_i' &:= E_i^{(1)}\,\neut_i\\
C_i' &:= C_{i},\\
P_i' &:= P_i\circ(E_i^{(1)})^{\inv},\\
E_i' &:= E_{i+1}^{(1)}\circ E_i^{(2)}.
\end{align}
In diagram, this looks like Figure \ref{FI structure composed es}. Then $\LambdaU{S}{i} = E_i^{(1)}$ and $\LambdaU{C}{i} = \Id$ form an isomorphism from $\Psi$ to $\Psi'$.
\end{theorem}
\begin{proof}
We observe that our $\LambdaE{S}$ and $\LambdaE{C}$ are bijective, satisfying \eqref{EQ hom neut}, and thus the claim follows from Theorem \ref{TH isom es}.
\end{proof}

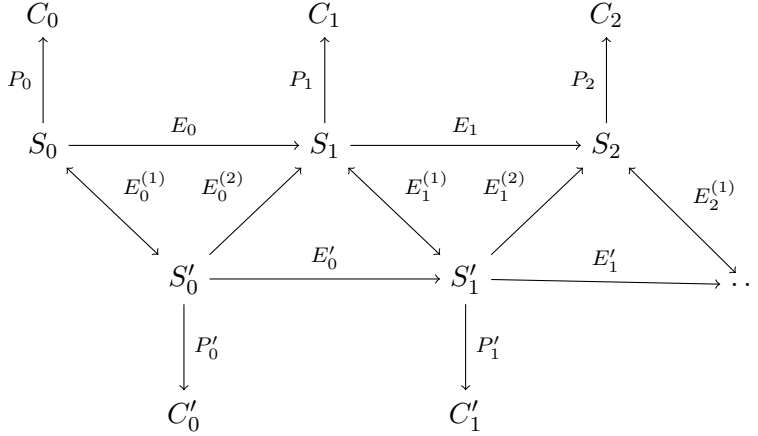
\begin{figure}[h]
\begin{center}
\begin{tikzpicture}[descr/.style={fill=white,inner sep=2.5pt}]
 \matrix (m) [matrix of math nodes, row sep=3em, column sep=3em,ampersand replacement=\&]
 {C_0 \& \ \& C_1 \& \ \& C_2 \& \ \\
  S_0 \& \ \& S_1 \& \ \& S_2 \& \ \\
  \ \& S_0' \& \ \& S_1' \& \ \& \ldots \\
  \ \& C_0' \& \ \& C_1' \& \ \& \ \& \\};
 \path[->,font=\scriptsize]
    (m-2-1) edge node[auto] {$ P_0 $} (m-1-1)
            edge node[auto] {$ E_0 $} (m-2-3)
    (m-3-2) edge node[auto] {$ P_0' $} (m-4-2)
            edge node[auto] {$ E_0' $} (m-3-4)
            edge node[auto] {$ E_0^{(2)} $} (m-2-3)
    (m-2-3) edge node[auto] {$ P_1 $} (m-1-3)
            edge node[auto] {$ E_1 $} (m-2-5)
    (m-3-4) edge node[auto] {$ P_1' $} (m-4-4)
            edge node[auto] {$ E_1' $} (m-3-6)
            edge node[auto] {$ E_1^{(2)} $} (m-2-5)
    (m-2-5) edge node[auto] {$ P_2 $} (m-1-5);
 \path[<->,font=\scriptsize]
    (m-2-1) edge node[auto] {$ E_0^{(1)} $} (m-3-2)
    (m-2-3) edge node[auto] {$ E_1^{(1)} $} (m-3-4)
    (m-2-5) edge node[auto] {$ E_2^{(1)} $} (m-3-6);
\end{tikzpicture}
\end{center}
\caption{Expansion systems with a relative shift}
\label{FI structure composed es}
\end{figure}

\subsection{Applications of Isomorphisms}

We will now continue with some examples of isomorphisms between expansion systems.

\begin{example}[\bf Newton's interpolation formula]\label{EX newton dif eq}
Consider an ES $\Psi$ where each $S_i$ is the space of real polynomials, which we will view as functions on $\R$ in the variable $x$, and we let $\neut_i=0$. Let $P_i$ be given by $P_i(y)=y(0)$, and let $E_i:=\Delta$, which is the forward difference operator: $\Delta y(x)=y(x+1)-y(x)$. Now $F_i=(P_i,E_i)$ is injective, for suppose we would have $y(0)=z(0)$ and $\Delta(y)=\Delta(z)$ for certain polynomials $y,z\in S_i$. Then by linearity $\Delta(y-z)=0$ and $(y-z)(0)=0$, which is only satisfied by the zero polynomial. Hence $y=z$, and we conclude that $\Psi$ is proper. As we have that $\Delta^k y=0$ iff $k$ is bigger than the polynomial degree of $y$, Definition \ref{DF order} tells us that the order of an element $y$ equals its degree plus one (except for the order of $0$, which is just $0$). So all elements in $S_0$ are elements of finite order, and thus by Remark \ref{RM conv wrt fin order} we have that $\Psi$ is an accurate ES, regardless of the topology on $S_0$.

However, this theoretic result doesn't give us an idea of how these kind of approximations look like. To find this out, let's not write out our polynomials in the more common basis $1,x,x^2,x^3,\text{etc.}$, but in the base
$$
\binom{x}{0},\binom{x}{1},\binom{x}{2},\ldots=1,\frac{x}{1!},\frac{x(x-1)}{2!},\frac{x(x-1)(x-2)}{3!},\ldots,
$$
where the $\binom{x}{k}$ denotes a binomial coefficient. Now if for certain $i$, $y_{i+1}^{[n]}$ happens to be of the form
$$
y_{i+1}^{[n]}=c_{i+1}\binom{x}{0}+\ldots+ c_{n-1}\binom{x}{n\,\shortmin\,i\,\shortmin\,2},
$$
where $c_i=\acP_i\,y$, then using that $\Delta \binom{x}{k}=\binom{x}{k-1}$, we get:
$$
{F_i}^{\inv}(c_i,y_{i+1}^{[n]})=c_i+c_{i+1}\binom{x}{1}+\ldots+ c_{n-1}\binom{x}{n\,\shortmin\,i\,\shortmin\,1}.
$$
That $y_{i+1}^{[n]}$ is indeed expressible in the form above, follows by induction, whereby we observe that $y_{n}^{[n]}$ fits in this pattern as being the empty sum $0$. The $n$-th order approximation $y^{[n]}$ thus equals:
$$
y^{[n]}=c_0+c_{1}\binom{x}{1}+\ldots+ c_{n-1}\binom{x}{n\,\shortmin\,1}.
$$
As we had already proven that $\Psi$ is convergent, it follows that every polynomial $y$ can be written as:
$$
y = \sum_{k\geq 0} c_k \binom{x}{k} = \sum_{k\geq 0} \frac{\Delta^{\mop k}y(0)}{k!} \;x\cdots(x-k+1).
$$
This result is known as Newton's forward difference interpolation formula (cf. \cite[\S 1.2]{Hi}).

No isomorphisms so far. Now let's construct a new ES $\Psi'$, which element and coefficient spaces are exact copies of $S$ and $C$. The maps
\begin{align}
\LambdaU{S}{i}\,y(x) &= -y(-x),\quad\text{and}\\
\LambdaU{C}{i}\,c &= -c,
\end{align}
are then clearly automorphisms on $S_i$ and $C_i$ (they are even self-inverse functions), and $\LambdaE{S}$ satisfies \eqref{EQ hom neut}. Therefore, Corollary \ref{CR autom isom} tells us that if we put:
\begin{eqnarray*}
P'_i\,y(x) = \LambdaU{C}{i}\,P_i\,(\LambdaU{S}{i})^{\inv}\,y(x) = \LambdaU{C}{i}\,P_i\,-y(-x) =\\
\LambdaU{C}{i}\,-y(0) = y(0),
\end{eqnarray*}
and furthermore
\begin{eqnarray*}
E_i'\,y(x) = \LambdaU{S}{i+1}E_i(\LambdaU{S}{i})^{\inv} y(x) = \LambdaU{S}{i+1}E_i\, -y(-x)
=\\ \LambdaU{S}{i+1} (-y(-x-1)+y(-x)) = y(x)-y(x-1),
\end{eqnarray*}
then $\Psi'=(S,C,F')$ will be isomorphic with $\Psi$. Note that we here have $P'=P$, and $E_i'$ equals the backward difference operator $\nabla$. So the only difference with $\Psi$, is that our expansions are given by backward differences, instead of forward differences. To establish properness and accurateness of $\Psi'$, and to find out what approximations it provides, although we could, we do not have to repeat all the steps we have gone through for $\Psi$. Simply by invoking Corollary \ref{CR isom} and Remark \ref{RM isom prop}, we conclude that $\Psi'$ must be proper. Giving both $S_0$ and $S_0'$ the discrete topology, $\LambdaU{S}{0}$ becomes a homeomorphism. As we have already remarked, $\Psi$ is accurate for every topology on $S_0$, in particular the discrete topology. Therefore, $\Psi'$ has to be accurate as well, and for its approximations we find:
\begin{eqnarray}
{y(x)}^{[\Psi',n]}=\LambdaU{S}{0}\,((\LambdaU{S}{0})^{\inv}\,y)^{[\Psi,n]}=
\LambdaU{S}{0}\,(-y(-x))^{[\Psi,n]} =\notag\\
\LambdaU{S}{0}\,\Bigl(\sum_{k=0,\ldots,n-1} \frac{(-1)^{k+1}\nabla^{\mop k}y(0)}{k!} \;x\cdots(x-k+1)\Bigr) =\label{EQ delta nabla}\\
\sum_{k=0,\ldots,n-1} \frac{(-1)^{k+1}\nabla^{\mop k}y(0)}{k!} \;(-x)\cdots(-x-k+1) =\notag\\
\sum_{k=0,\ldots,n-1} \frac{\nabla^{\mop k}y(0)}{k!} \;x\cdots(x+k-1)\label{EQ pre newton back dif}
\end{eqnarray}
where in \eqref{EQ delta nabla} we use both that $\Delta(-y) = -\Delta y$ and that $\Delta(y(-x))=-(\nabla y)(-x)$, which implies by induction:
$\Delta^{\mop k}(-y(-x)) = (-1)^{k+1}(\nabla^k y)(-x)$. Taking the limit of \eqref{EQ pre newton back dif} as $n$ goes to infinity, we get:
\begin{equation}\label{EQ newton back dif}
y = \sum_{k\geq 0} \frac{\nabla^{\mop k}y(0)}{k!} \;x\cdots(x+k-1),
\end{equation}
which is known as Newton's backward difference interpolation formula (cf. \cite[Ch.~1,~prob.~6]{Hi}).
\end{example}

\begin{example}
The expansions of the decimal expansion, described in Example \ref{EX dec exp 2}, are composed of two maps $E_i^{(1)}$ and $E_i^{(2)}$, namely:
\begin{align*}
E_i^{(1)}&\colon [0,1)\rightarrow [0,10),\ y \mapsto 10y,\\
E_i^{(2)}&\colon [0,10)\rightarrow [0,1),\ y' \mapsto  y'-\lfloor y'\rfloor.
\end{align*}
As the first of these two maps is a bijection, we have by Theorem \ref{TH isom shift} that the $\Psi$ of Example \ref{EX dec exp 2} is isomorphic with the ES described by: $S_i' = [0,10)$, {$C_i'=C_i$}, $P_i\,y' = \lfloor y'\rfloor$ and $E_i'\,y' = 10(y'-P_i'\,y)$. The elements and its convergents from the former ES can be translated into elements and convergents for this ES, just by multiplying with a factor $10$.

Something similar applies to the continued fraction (Example \ref{EX cont frac}), and we have that the ES is isomorphic to one defined by: $S_i' = \R_{>1}$, $C_i'=C_i$, $P_i\,y' = \lfloor y'\rfloor$ and $E_i'\,y' = 1/(y'-P_i'\,y)$.

There are quite some more situations, for which there exists the choice in what order to apply the mappings which the expansions $E_i$ our composed of. In those cases, Theorem \ref{TH isom shift} may tell us that the chosen order is actually not essential for the approximations which the ES provides. Another example of this can be found in the appendix about approximation systems.
\end{example}

\newpage

\section*{Appendix: Approximation Systems}
\addcontentsline{toc}{section}{Appendix: Approximation Systems}

In this appendix, we will have a glimpse at so-called {\em approximation systems}\index{Approximation system}. A first description of approximation systems was given in \cite{Pe}, and in more detail it is discussed in \cite{PeKo}. Here, we will give another description of AS's (approximation systems), but now in the language of expansion systems. The research regarding approximation systems is still going on, and there are quite some issues that need yet to be clarified. The content of this appendix should therefore be read as a rough account of recent investigations, rather than a rigorous treatment of the subject.

Although the original formulation of AS's is somewhat different from the language of ES's, we will see that ES's are indeed capable of describing AS's, and even in a more general setting. However, this translation from AS's to ES's is not so straightforward, and particular complexities do arise. The different nature of AS's and ES's is the main reason for this. While the first is described by a sequence of differential equations for a particular sequence of functions, the second one has to be described by a sequence of maps between function spaces as a whole.

Perhaps we could avoid some problems by drastically limiting the extent of the spaces $S_i$, but besides that this would still leave some obstacles to overcome,
it would also be a pity, as it turns out that the use of expansion systems allows us to give a general approximation algorithm, which can be applied to a wide range of function at the same time. The formulation of AS's chosen in this section is based on function germs. This solves certain technical problems, that could otherwise prevent the ES from being well-defined. But presumably, it is not strictly necessary to rely on the concept of germs, and at the end of this section we will briefly discuss the possibility of describing the element spaces of an approximation system as consisting of analytic functions on a certain domain.

Instead of giving a complete definition of AS's, we will start give some examples of it. These examples will should already give an idea of what AS's are.

\subsection*{Examples of Approximation Systems}
\addcontentsline{toc}{subsection}{Examples of Approximation Systems}

Before we come to the first examples, we mention some notation regarding {\em function germs}, which will be used in the examples. For a brief description of germs, see for instance \cite[end of Ch.~3]{Le}.

\begin{definition}
Given an analytic germ $y$ at $x_0$, we denote the multiplicity of its zero in $x_0$ by $M(y)$. This can attain all values in $\Z_{\geq 0}$, as well as $\infty$, which we reserve for the germ constant $0$.
\end{definition}

\begin{example}\label{EX app sys 1}
Let $S_i$ be the space of germs in $x_0$ with constant term $1$, and let $\neut_i=1$. Further, we let $C_i = \C \times (\Z_{\geq 0}\cup \{\infty\})$, and let $P_i$ be given by
$$
P_i\,y=(\;\left[(x-x_0)^{-M(D y)}D y\right]\!(x_0)\;,\;M(D y)\;),
$$
where the first entry simply denotes the first Taylor coefficient after the constant term $1$. If $y$ is constant $1$, we may set $P_i\,y = (0,\infty)$ and $E_i\,y = 1$ (satisfying \eqref{EQ neut elt}). For all other $y$ we set:
\begin{equation}\label{EQ as 1 E_i}
E_i(y) = \left(\frac{Dy(x)}{c_i (x-x_0)^{m_i}}\right)^{\alpha_i},
\end{equation}
where $(c_i,m_i)=P_i\,y$, and $\alpha_i$ can be any value in $\C\setminus \{0\}$ (note that the germ $h:=Dy(x)/(c_i (x-x_0)^{m_i})$ is $1$ in $x_0$, so that we can take the principle value for the logarithm in $h^{\alpha_i} = \exp({\alpha_i} \log\,h)$~). These mappings define a proper ES $\Psi$, for we have that $F_i=(P_i,E_i)$ is bijective, with inverse:
$$
{F_i}^{\inv}(c,m,y) = 
1+\displaystyle\int_{x_0}^x c (t-x_0)^m y(t)^\frac{1}{\alpha_i}\,dt,
$$
which evaluates to $1$ in case $c=0$. In the context of our ES, all these expressions are formally to be interpreted as just germs in $x_0$, but as we will see later, they have more practical value than just representing germs.
\end{example}

\begin{remark}
The expansions $E_i$ can be decomposed into $E_i^{(1)}\colon S_i\rightarrow S_i'$ and $E_i^{(2)}\colon S_i'\rightarrow S_{i+1}$, where:
\begin{align}
E_i^{(1)}\, y &= D\,y,\label{EQ decomp E 1}\\
E_i^{(2)}\, y' &=\left(\frac{y'(x)}{c_i (x-x_0)^{m_i}}\right)^{\alpha_i}.\label{EQ decomp E 2}
\end{align}
As the first of these maps is bijective on the space of germs with constant term $1$, we have by Theorem \ref{TH isom shift} that our $\Psi$ is isomorphic with $\Psi'$, where:
\begin{align}
P_i'\,y &= (\;\left[(x-x_0)^{-M(y)}y\right]\!(x_0)\;,\;M(y)\;),\\
E_i'\,y(x) &= D_x\,\left(\frac{y(x)}{\left[(x-x_0)^{M(y)}y\right](x_0)}\right)^{\alpha_i}.
\end{align}
Note that we have:
\begin{equation}\label{EQ comp P_i}
P_i = P_i'\,E_i^{(1)}.
\end{equation}
In fact, the formulation of $\Psi'$ has been the underlying principle in \cite{Pe}, whereas the formulation of $\Psi$ coincides with the formulation in \cite{PeKo}. As $\LambdaE{S}\,y = D y$ gives an isomorphism between the spaces $S_i$ and $S_i'$, and $S_i'$ does not have the restriction of having constant term $1$, it is clear that this restriction on $S_i$ is actually not that limitative. In fact we could drop this restriction on $S_0$ (i.e. we take as initial space $S_0'$), and in a similar way as we have mentioned in Example \ref{EX dec exp 2}, we let our first expansion be given by simply $E_0^{(2)}$, keeping all further expansions just as they were. Then the ES can be applied to any holomorphic germs.
\end{remark}

As we haven't fixed the value of $\alpha_i$ in the AS in Example \ref{EX app sys 1}, it actually represents a whole collection of AS's. We will now consider some particular instances of it. When each $\alpha_i=1$, it can easily be verified that the approximations are just the Taylor approximations (besides the restriction in $S_0$ on the constant term, the main difference with the approximations provided by Example \ref{EX taylor}, is that here the $n$-th order convergent gives the first $n-1$ non-zero terms of the Taylor series, whereas in Example \ref{EX taylor} the terms up to and including order $n-1$ will be given).

More generally, when each $\alpha_i$ is a positive unitary fraction $1/p_i$ with $p_i\in \Z_{\geq 1}$, then the resulting approximations will be polynomials. Using concise notation, such as we have used in \eqref{EQ cont frac} and \eqref{EQ eng exp}, we can write the approximation process as:
\begin{equation}\label{EQ as unit frac}
y(w) = 1+\displaystyle \int_{0}^{w}  c_0 {w_1}^{m_0}\left(1+\displaystyle\int_{0}^{w_1}  c_1 {w_2}^{m_1} \Bigl(\ldots\Bigr)^{p_1}
dw_2\right)^{p_0}dw_1,
\end{equation}
using local coordinates $w = x-x_0$, and where $(m_i,c_i)= \acP_i\,y$. This expression suggest that this ES is actually accurate. Of course, by Theorem \ref{TH conv weakly discrete top} we have convergence in the weakly discrete topology. In fact, it can be shown that this means that the first $n-1$ non-zero Taylor coefficients of the $n$-th convergent coincide with the first $n-1$ non-zero Taylor coefficients of the initial function $y$ (cf. \cite[Cr.~5.4]{PeKo}). This means we have also weak convergence in the sense that the derivatives of all orders convergence in the point $x_0$. But in practice we often have even more than that. Consider the situation in which each $p_i=2$, and look at the element $y(x)=1/x$ and $x_0=1$. Then \eqref{EQ as unit frac} reduces to:
\begin{equation}\label{EQ 1/sqrt(1-x)}
\frac{1}{\sqrt{1-x}} = 1+\frac{1}{2}\displaystyle\int_{0}^{x} \left(1+\frac{3}{4}\displaystyle\int_{0}^{x_1} \left(1+\frac{7}{8}\displaystyle\int_{0}^{x_2} \Bigl(\ldots\Bigr)^2dx_3\right)^2
dx_2\right)^{2}dx_1.
\end{equation}
The coefficients $c_i$ in this formula, have been moved outside of the integrals. The first few convergents are given by:
\begin{align*}
y^{[0]}(x)&=1\\
y^{[1]}(x)&=1+\frac{x}{2}\\
y^{[2]}(x)&=1+\frac{x}{2}+\frac{3 x^2}{8}+\frac{3 x^3}{32}\\
y^{[3]}(x)&=1+\frac{x}{2}+\frac{3 x^2}{8}+\frac{5 x^3}{16}+\frac{175 x^4}{1024}+\frac{147 x^5}{2048}+\frac{343 x^5}{16384}+\frac{343 x^5}{131072}
\end{align*}
Interpreting these polynomials as complex functions rather than germs, convergence can be established on the open disk with radius $1$. In \cite{PeKo} this has been proven for all powers $p\geq 2$. However, despite the fact that the head of all these polynomials coincides with the Taylor series, computer simulations have pointed out that the actual domain of convergence is in fact larger than this open disk, and thus not symmetric. Of course, the domain of convergence {\em is} symmetric in the complex direction, and the pole in $x=1$ lies on its boundary. What the exact contours are, remains nonetheless a challenging question.

\begin{remark}
In the example above, if $S_i$ would be defined as a space of holomorphic functions on a fixed neighborhood $U$ of $x_0$, then we can always find a function $y$ on which $E_i$ is not well-defined, as it would introduce a branch point to $U$. And at the same time, possibly this $y$ {\em could} be successfully approximated, if only our domain was somewhat smaller than $U$. In the formulation with germs, the algorithm applies without problems to all relevant functions. In the subsection {\em Further Aspects of Approximation Systems} we will briefly discuss an alternative formulation, which could possibly solve this issue in a yet different way.
\end{remark}

We have allowed the $\alpha_i$ in Example \ref{EX app sys 1} to depend on $i$, and it is indeed possible to vary $\alpha_i$ as $i$ increases. For example, when we let $\alpha_i$ run through $\frac1 2,\frac1 3,\frac1 4,\ldots$ and we look at $y(x) = e^x$, then we get the expansion:
\begin{equation}\label{EQ exp powers}
e^x = 1+\int_0^x \left(1+\frac{1}{2!}\int_0^{x_1} \left(1+\frac{1}{3!}\int_0^{x_2} \Bigl(\ldots\Bigr)^4
dx_3\right)^3dx_2\right)^2dx_1.
\end{equation}
Although the convergence of this formula turns out to be very fast, it also requires much more effort to evaluate than a Taylor series.

So far we have only considered unit fractions, which give polynomial convergents. Another possibility are positive integer $\alpha_i$. For instance, when we let $\alpha_i$ run through $2,3,4,\ldots$ instead, and pick $y=e^x$ again, then we get more or less the opposite of \eqref{EQ exp powers}:
\begin{equation}
e^x = 1+\int_0^x \sqrt{1+2!\int_0^{x_1} \sqrt[3]{1+3!\int_0^{x_2} \displaystyle\sqrt[4]{\Bigl. \ldots\Bigr. }
\ dx_3}\ dx_2}\ dx_1.
\end{equation}
Here, the convergents are no longer polynomials, but more complex expressions. Symbolic computation is impossible for higher orders of the convergents, and so the only option seems to be numerical evaluation, or computations in terms of power series. Although the appearance of the factorials in the last two formulas for $e^x$ may seem remarkable, they are actually imposed by our choice of $\alpha_i$. For example, if we would keep $\alpha_i$ constant as in \eqref{EQ 1/sqrt(1-x)}, then the coefficients form a geometric sequence instead.

Another interesting instance of Example \ref{EX app sys 1}, we get by setting $\alpha_i=-1$. In concise notation, this gives us expansions like:
\begin{align}
e^x\quad &=\quad 
1+\displaystyle\int_0^x\cfrac{1}{
1-\displaystyle\int_0^{x_1}\cfrac{dx_2}{
1+\displaystyle\int_0^{x_2}\cfrac{dx_3}{
1-\displaystyle\int_0^{x_3}\cfrac{\quad\; dx_4\quad\;}{\ldots}}}}\label{EQ exp x cont frac as}\\
x^a \quad&=\quad
1+\displaystyle\int_1^x\cfrac{a\; dx_1}{
1+\displaystyle\int_1^{x_1}\cfrac{1-a\; dx_2}{
1+\displaystyle\int_1^{x_2}\cfrac{a\; dx_3}{
1+\displaystyle\int_1^{x_3}\cfrac{\ 1-a\; dx_4\ }{\ldots}
}}}\label{EQ x^a}\\
\cosh\ x\quad &=\quad
1+\displaystyle\int_0^{x}\cfrac{x_1\,dx_1}{
1-\displaystyle\int_0^{x_1}\cfrac{x_2\,dx_2}{
3+\displaystyle\int_0^{x_2}\cfrac{7\,x_3\,dx_3}{
5-\displaystyle\int_0^{x_3}\cfrac{221\,x_4\,dx_4}{2205\cdot \ldots}
}}}
\end{align}

\begin{remark}
The first two of these formulas show a regular pattern. In fact, we can view these two formulas as a repeated application of an operator $T$, for which $y$ is a fixed point. For instance, in \eqref{EQ exp x cont frac as} we have:
$$
T y(x) =
1+\displaystyle\int_0^x\cfrac{dx_1}{
1-\displaystyle\int_0^{x_1}\frac{dx_2}{y(x_2)}}
$$
which has as fixed point $e^x$. The reason that there exists such an operator $T$, is that the initial function $y$ satisfies a finite cycle of differential equations, determined by \eqref{EQ as 1 E_i}. In this case, it is a cycle of length $2$. In general, if we are given a cycle of equations:
\begin{equation}
y_{i}(x) = c_i+\int_{x_0}^x f_i(y_{i+1}(t),t)\,dt,\quad (i=0,\ldots,n-1,\quad y_n = y_0),
\end{equation}
then the solution may also be found by a Picard iteration:
$$T\,{\vec y} = \vec c + \int_{x_0}^x \vec f \circ \vec y(t)\,dt,$$
on a multi-dimensional $\vec y:\C\rightarrow \C^n$ (cf. \cite[7.2]{HiSmDe}). For this, we set $\vec f:\C^n\rightarrow \C^n$ as:
$$\vec f(y_0,\ldots,y_{n-1}) = (f_0\circ y_1,\ldots,f_{n-2}\circ y_{n-1},f_{n-1}\circ y_0)$$ and as constant of integration simply $\vec c = (c_0,c_1,\ldots,c_{n-1})$.

Starting this iteration with the vector $\vec y = (\neut_0,\ldots,\neut_{n-1})$ (for the case each $\neut_i=1$, this is just the vector $(1,\ldots,1)$), then after $n$ repetitions, the first entry of this vector equals $y^{[n]}$.
\end{remark}

Of course, the algorithm still applies when there is no such cycle for $y$. The third formula is an example of this, and its coefficients do not follow an obvious pattern. Such continued fraction-like expressions can be obtained for every analytic function, and for most it will require quite some effort to calculate the coefficients. Furthermore, the convergents cannot be expressed in terms of elementary functions for high orders of $n$. For example, in \eqref{EQ exp x cont frac as} we can only determine the convergents up to and including order $4$:
\begin{align*}
y^{[0]}(x)&=1\\
y^{[1]}(x)&=1+x\\
y^{[2]}(x)&=1-\log(1-x)\\
y^{[3]}(x)&=1+e\cdot\li\,\frac{1}{e}-e\cdot\li\,\frac{1+x}{e},
\end{align*}
where $\li$ denotes the logarithmic integral. For the orders $n\geq 4$, we have to rely on numerical methods to evaluate the expression. Although the rate of convergence of the above formulas turns out to be rather slow, its domain of convergence is surprisingly large. In the situation of \eqref{EQ x^a}, the approximations even seem to converge when we evaluate them on paths going around the branch point of $x^a$, when $a$ is not a nonnegative integer. We will come back to this observation at the end of this appendix. Before that, let us consider two more examples of AS's.

\begin{example}\label{EX app sys 2}
Unlike Example \ref{EX app sys 1}, we now let $S_i$ be the space of holomorphic germs at $x_0$ with constant term $0$, and $\neut_i=0$. We let $C_i$ and $P_i$ be just as in Example \ref{EX app sys 1}, but now we set $E_i\,y = 0$ if $y=0$, and otherwise:
$$
E_i(y) = \log\, \frac{Dy(x)}{c_i x^{m_i}}.
$$
These mappings define a proper ES $\Psi$, for we have that $F_i=(P_i,E_i)$ is bijective, with inverse:
$$
{F_i}^{\inv}(c,m,y) = 
\int_{x_0}^x c (x-x_0)^m e^{y(t)}\,dt,
$$
which evaluates to $0$ in case $c=0$.
\end{example}

This AS provides us with formulas such as:
\begin{align}
\log\,x\quad &=\quad
\medint\int_{1}^{x  }\displaystyle dx_1\,e^{-
\medint\int_{1}^{x_1}\displaystyle dx_2\,e^{-
\medint\int_{1}^{x_2}\displaystyle dx_3\,e^{-
\medint\int_{1}^{x_3}\displaystyle dx_4\,e^{\ldots}}}}\label{EQ exp AS 1}\\
\tan\, x\quad &=\quad
\medint\int_{0}^{x  }\displaystyle dx_1\, e^{
\medint\int_{0}^{x_1}\displaystyle dx_2\,2x_2 \,e^{
\medint\int_{0}^{x_2}\displaystyle dx_3\,\textstyle\frac{2}{3}x_3 \,e^{
\medint\int_{0}^{x_3}\displaystyle dx_4\,\textstyle\frac{14}{15}x_4 \,e^{\ldots}}}}\label{EQ exp AS 2}
\end{align}
For clarity, we put the variable of integration in front of the integrand, instead of behind. The first of these formulas can be viewed as iterating a single operator, like we saw with \eqref{EQ exp x cont frac as} and \eqref{EQ x^a}. In this case, the operator entails only one integration in total, as the cycle of differential equations consists merely of the equation:
$$
\frac{d}{dx}\,-\log x = -e^{-\log x}.
$$
In other words, this formula just expresses a Picard iteration. Note that in Equation \ref{EQ exp AS 1}, we have multiplied both sides of the equation by a factor $-1$ afterwards. The second formula does not arise from a cycle of differential equations, and so the coefficient do not follow a cycle either. Numerical simulations suggest that both of these formulas converge on a certain domain (as for \eqref{EQ exp AS 1}, which is equivalent to a Picard iteration, this should not come as a surprise).

\begin{example}\label{EX app sys 3}
In this example, we let $S_i$ consist of the germs with constant term $1$ again, as in Example \ref{EX app sys 1}, but now we let $C_i = \C^2 \times (\Z_{\geq 0}\cup \{\infty\})$. Further, we let $E_i$ be the composition of maps $E_i^{(1)}$ and $E_i^{(2)}$, where $E_i^{(2)}$ is defined as in \eqref{EQ decomp E 2}, and $E_i^{(1)}$ is the composition $K\circ D$, where $K\,y := y - y(x_0)$. This basically means that $E_i^{(1)}$ first differentiates a function, and chops off the constant term afterwards. Note that this $E_i^{(1)}$ isn't bijective any longer. We set $P_i$ to:
$$
P_i\,y=(Dy(x_0),(x^{-M(K\,D y)}K\,D y)(x_0),M(K\,D y)).
$$
These mappings define a proper ES $\Psi$, and $F_i=(P_i,E_i)$ is bijective, with inverse:
$$
{F_i}^{\inv}(b,c,m,y) = 
1+b x+ \displaystyle\int_{x_0}^x c (t-x_0)^m y(t)^\frac{1}{\alpha_i}\,dt.
$$
We may verify that this expression is mapped to $y$ by $E_i$, and that the projection $P_i$ indeed gives the parameters $(b,c,m)$.
\end{example}

The above AS gives us among others the expansion:

\begin{align}\label{EQ as chop}
&3x^2+x^3 \quad =\notag\\
&\int_{0}^{x  }6x_1
\sqrt[3]{1+\frac{3x_1}{2}+\int_{0}^{x_1}\frac{3x_2}{2}
\sqrt[3]{1+\frac{3x_2}{4}+\int_{0}^{x_2}\frac{3x_3}{8}
\sqrt[3]{\bigl.\ldots\bigr.}\ dx_3}\ dx_2}\ dx_1.
\end{align}

This equation is obtained by applying the AS to $(x+1)^3$, and subtracting the equal terms (i.e. $1+3x$) afterwards. The coefficients obey the pattern $b_i = 3/2^i$ and $c_i = 3/2^{2i-1}$.

\subsection*{Further Aspects of Approximation Systems}
\addcontentsline{toc}{subsection}{Further Aspects of Approximation Systems}

We have just seen some examples of approximation systems, as well as the kind of results it gives. This may have sparked a couple of questions. First of all, what is the meaning of all these formula's? As noted before, when we interpret the equations as dealing with germs, the equality symbol expresses that the left hand side is the limit of a sequence defined by the right hand side, namely in the weakly discrete topology.
But in many cases, the expressions do also converge in the more common topology. For some particular cases, this has been established in \cite{PeKo}. For many others, this has only been suggested by numerical simulations.

Although for some AS's the rate of convergence can be really high with respect to particular elements $y$, in most cases this does not seem to compensate for the effort necessary to (numerically) evaluate these kind of expansions. Of course, it may well be that there are situations in which the algorithm really serves practical purpose, but in other situations it would probably not be the prime motivation to study these identities. Probably more interesting is the fact that these expansion do converge in the first place, also when it involves quickly increasing functions such as in \eqref{EQ exp AS 2}. Even more remarkable, is that the domain of convergence can take extraordinary shapes. We already mentioned this regarding \eqref{EQ 1/sqrt(1-x)} and \eqref{EQ x^a}.

For the first of these, we know for sure that its domain of convergence is at least the open disk with radius 1 (see \cite{PeKo}), but evaluating the algorithm numerically, we find a domain which is considerably larger than that, delimited by a very irregular boundary, with even fractal like characteristics. As for now, we could only speculate about what this is supposed to mean.

At Equation \eqref{EQ x^a}, we already mentioned that here the approximations seem to converge, even when going around the singularity at $x=0$ for $a\notin \Z_{\geq 0}$. What we mean with this is the following. Let's first choose a parametrization $\gamma\colon [0,r]\rightarrow \C\setminus\{0\}$ of a path going around the origin (this could be $\gamma(t)=e^{t i}$ for example, but if we are performing a numerical simulation, a path in the shape of a square would also do). Next, we evaluate the expansion given in \eqref{EQ x^a} up to a certain order, but when calculating the integrals, we restrict ourselves to the path $\gamma$. Then, we determine the values of $x^a$ on the same path $\gamma$ by analytic continuation (if $\gamma(t)=e^{t i}$, this is just $e^{a t i}$). If our $\gamma$ doesn't happen to have crossed a singularity of some $y_i^{[n]}$ (of which there can only be finitely many), it turns out that the approximated values of $y^{[n]}$ on $\gamma$, converges to the value of $y$ on this $\gamma$, as $n$ gets larger. In somewhat loose notation, this means:
$$
\gamma(t)^a \quad=\quad
1+\displaystyle\int_0^t\cfrac{a\;D\gamma(t_1)\,dt_1}{
1+\displaystyle\int_0^{t_1}\cfrac{1-a\;D\gamma(t_2)\,dt_2}{
1+\displaystyle\int_0^{t_2}\cfrac{a\;D\gamma(t_3)\,dt_3}{
1+\displaystyle\int_0^{t_3}\cfrac{\ 1-a\;D\gamma(t_4)\,dt_4\ }{\ldots}
}}}
$$
where the integration variables $t_i$ run through the domain $[0,r]$. As long as $\gamma$ doesn't accidentally cross some singularity like we just mentioned, this identity always seems to hold, no matter what the end point of $\gamma$ is, and how many times it has turned around the singularity at $x=0$. Performing a similar procedure on the Taylor expansion (viewed as an AS), we don't have this property. For example, picking $\gamma(t)=e^{t i}$, we can find either by calculation or by reasoning, that:
$$
\gamma(t)^a = 1+a (e^{i t}-1)+\frac{a(a-1)}{2!} (e^{i t}-1)^2+\ldots
$$
which diverges for $t>\frac{\pi}{3}$.

In order to study this kind of convergence behavior for AS's in general, it may be desirable to give a different formulation of AS's. Possibly, we could let the spaces $S_i$ consist of multilayered functions covering the complex plane, allowing different functions to live on different domains, but each of these domains being marked with a certain basepoint $x_0$. Then the expansions $E_i$ determine a transition from one multilayered function to another, which might as well alter the Riemann surface a function lives on. If we want to compare the values of two functions $y$ and $y'$ that live on different domains, (so that we are able to define a criterion for convergence), we can lift paths in $\C$, such that the begin point of the lift is in both situations $x_0$. The difference in value of $y$ and $y'$ on the lift of a suitable path $\gamma$, serves as a measure for how close the two functions are on that particular path. Of course, to give this whole formulation a solid foundation, will require a lot more work, which we leave outside the scope of this thesis.

After all this discussion, there is at least one question which still isn't answered yet: what is an approximation system? The different examples may of course have given an impression of what is meant with it, but a strict definition has not been given. As one may have noticed, the expansions and projections take various forms, which are difficult to capture all in a single definition. Restricting ourselves to the formulation with germs, approximation systems may be informally described as expansion systems on spaces of holomorphic germs, for which each expansion is a transformation which (possibly) erases some of the information of the lower order derivatives (i.e. an initial part of the Taylor series), and for which the projections are the containers of this information. A further smoothness condition on the expansions and projections may be desirable, in order to have convergence of the AS. Of course, this description of AS's is still to vague to be used within a mathematical context, but hopefully it does give an idea of the mechanism behind AS's.

It is clear that approximation systems open the door to a virtually endless number of expansions. And then we haven't even mentioned yet the case in which we do not involve any differentiations in our maps $E_i$. For instance, if we replace the differential operator in Example \ref{EX app sys 1} and Example \ref{EX app sys 2} by the operation $K\,y = y-y(x_0)$, as introduced in Example \ref{EX app sys 3}, then a whole new family of identities arises, such as:
\begin{align}
W(x)\quad &=\quad x
e^{\displaystyle -x
e^{\displaystyle -x
e^{\displaystyle -x
e^{\displaystyle -x
e^{\ldots}}}}}\label{EQ K exp W}\\
\ \notag\\
2+x \quad &= \quad 2
\sqrt{1+x
\sqrt{1+\frac{x}{2}
\sqrt{1+\frac{x}{4}
\sqrt{1+\frac{x}{8}
\sqrt{\bigl.\ldots\bigr.}}}}}\label{EQ K sqrt}\\
\tan\,x\quad &=\quad\notag
\cfrac{x}{
1-\cfrac{x^2}{
3-\cfrac{x^2}{
5-\cfrac{x^2}{
7-\cfrac{\;x^2\;}{\ldots}}}}}\label{EQ K lambert cont frac}\\
\end{align}
where $W(x)$ denotes the Lambert $W$ function. In the literature, the first two kinds of expansions are respectively referred to as {\em infinite exponentials} (cf. \cite{Ba}), and {\em nested radicals} (cf. \cite{BoBa}). The third type of expansion is again a {\em continued fraction}, but different from Example \ref{EX cont frac}, now including a variable $x$. For arbitrary analytic functions, the approximation provided by this ES corresponds with the so-called {\em C-fraction} for a given analytic function, or even formal power series (cf. \cite[2.4]{HiSmDe}). This particular identity in \eqref{EQ K lambert cont frac} belongs to the class of {\em Gauss's continued fractions}, and is more specifically called {\em Lambert's continued fraction} (cf. \cite[91.7]{Wa}).

\paragraph{CONTACT DETAILS}\quad\\[\smallskipamount]
V.A. Pessers;\\ 
email:\ {\tt vpessers@gmail.com}

\newpage
\printindex

\end{document}